\documentclass[final]{siamart190516}

\usepackage{epsfig, graphicx}
\usepackage{latexsym,amsfonts,amsbsy,amssymb}
\usepackage{amsmath,amscd}
\usepackage{cite}
\usepackage{bm}
\usepackage{algorithm}
\usepackage{diagbox}
\usepackage{tikz}
\usetikzlibrary{matrix}
\usetikzlibrary{shapes,patterns,arrows,snakes,decorations.shapes}
\usetikzlibrary{positioning,fit,calc}
\usetikzlibrary{plotmarks}
\usepackage{pgfplots,pgfplotstable}
\pgfplotsset{width=10cm, compat=newest}

\newcommand{\dive}{{\rm div}}
\DeclareMathOperator{\ddiv}{div}

\renewcommand{\vec}[1]{\ensuremath{\boldsymbol{#1}}}

\def\Inf{\operatornamewithlimits{inf\vphantom{p}}}

\usepackage{xspace,color}

\title{Monolithic multigrid for a reduced-quadrature discretization of poroelasticity\thanks{Submitted to the editors on 6/25/2021.
		\funding{The work of JA, XH, and PO was partially funded by National Science Foundation grant DMS-1620063. The work of S.M. was partially funded by an NSERC Discovery Grant. The work of PO was supported by the U.S.~Department of Energy, Office of Science, Office of Advanced Scientific Computing Research, Applied Mathematics program.  Sandia National Laboratories is a multimission laboratory managed and operated by National Technology and Engineering Solutions of Sandia, LLC., a wholly owned subsidiary of Honeywell International, Inc., for the U.S. Department of Energy's National Nuclear Security Administration under grant DE-NA-0003525. This paper describes objective technical results and analysis. Any subjective views or opinions that might be expressed in the paper do not necessarily represent the views of the U.S. Department of Energy or the United States Government.
}}}

\author{James H. Adler\thanks{Department of Mathematics, Tufts University, Medford, MA 02155, USA
		(\email{james.adler@tufts.edu}, \email{xiaozhe.hu@tufts.edu}).}
	\and Yunhui He\footnotemark[4]\ \thanks{Department of Computer Science, The University of British Columbia, Vancouver, BC, V6T 1Z4, Canada (\email{yunhui.he@ubc.ca}).}
	\and Xiaozhe Hu\footnotemark[2]
	\and Scott MacLachlan\thanks{Department of Mathematics and Statistics, Memorial University of Newfoundland, St. John's, NL A1C 5S7, Canada
		(\email{yunhui.he@mun.ca}, \email{smaclachlan@mun.ca}).}
	\and Peter Ohm\thanks{Sandia National Laboratories, Albuquerque, NM 87185 (\email{pohm@sandia.gov})}}

\newsiamremark{remark}{Remark}

\begin{document}

\maketitle
\begin{abstract}
  Advanced finite-element discretizations and preconditioners for models of poroelasticity have attracted significant attention in recent years.  The equations of poroelasticity offer significant challenges in both areas, due to the potentially strong coupling between unknowns in the system, saddle-point structure, and the need to account for wide ranges of parameter values, including limiting behavior such as incompressible elasticity.  This paper was motivated by an attempt to develop monolithic multigrid preconditioners for the discretization developed in \cite{rodrigo2018new}; we show here why this is a difficult task and, as a result, we modify the discretization in \cite{rodrigo2018new} through the use of a reduced quadrature approximation, yielding a more ``solver-friendly'' discretization.  Local Fourier analysis is used to optimize parameters in the resulting monolithic multigrid method, allowing a fair comparison between the performance and costs of methods based on Vanka and Braess-Sarazin relaxation.  Numerical results are presented to validate the LFA predictions and demonstrate efficiency of the algorithms.  Finally, a comparison to existing block-factorization preconditioners is also given.
\end{abstract}



\section{Introduction}
Poroelasticity describes a number of processes modeled by flows in deformable porous media,
which are of interest in geoscience, biomedical science, and engineering.  In this paper,
we consider Biot's model for linear poroelasticity \cite{biot1,biot2}, a coupled, multiphysics system of partial differential equations (PDEs).
There are many challenges in developing both discretizations and fast and robust solvers for these equations.
For the discretization, using mixed finite elements, it is necessary to carefully choose approximation spaces in order to avoid
spurious oscillations in the pressure field as well as to achieve robustness to variations in the PDE parameters, particularly in extreme limits, such as incompressibility.
After discretization, the resulting linear system is of saddle-point type, requiring special solvers to deal with the indefiniteness, the usual
ill-conditioning of the discretized system, and to achieve similar robustness with respect to the physical parameters.

Many different types of discretizations exist for the various formulations of Biot's model.
For instance, a finite-volume method on a staggered grid is provided for the three-dimensional
Biot poroelastic system in \cite{naumovich2008multigrid}.
For the
two-field formulation, where displacement and pressure are the unknowns, stable Taylor-Hood elements are used in \cite{MuradLoulaThome,MuradLoula94,MuradLoula92}.  In \cite{rodrigo2016stability}, a MINI element and a stabilized P1-P1 finite-element discretization
are presented, and a stabilization term is introduced to remove non-physical oscillations,
leading to  monotone behavior of the stabilized schemes.  More recently, a weak Galerkin finite-element method is proposed on general shape-regular polytopal meshes,
which demonstrates the robustness of the proposed weak Galerkin discretization \cite{hu2018weak}.
For three-field formulations, where displacement, pressure, and the Darcy velocity are the unknowns,  a nonconforming finite-element approach for the three-field formulation, using Crouzeix-Raviart finite elements for the displacements,
lowest-order Raviart-Thomas-N\'{e}d\'{e}lec elements for the Darcy velocity,
and a piecewise-constant approximation for the pressure, is considered in~\cite{yiCouplingNonconformingMixed2013} for the two-dimensional case on rectangular grids.  It is extended to general cases in \cite{hu2017nonconforming}, where a mass-lumping technique is introduced for the Raviart-Thomas-N\'{e}d\'{e}lec elements to eliminate the Darcy velocity, reducing the computational cost. A family of parameter-robust schemes is found in \cite{hong2018parameter} and a general theory for the error analysis is introduced. More recently, hybridization schemes are developed in~\cite{NiuRuiHu2019,frigoEfficientSolversHybridized2020}.
For a four-field formulation, with the stress tensor, fluid flux, displacement, and pore pressure as unknowns, stable discretizations are developed in \cite{yiConvergenceAnalysisNew2014,lee_four_field}.

In this work, we consider a stabilized finite-element method based on the popular P1-RT0-P0 discretization of the three-field formulation developed in \cite{rodrigo2018new}, where face bubble functions are used to enrich the P1 space for the displacements.
A perturbation of the bilinear form allows for local elimination of the bubble functions,
leading to the same number of degrees of freedom as the P1-RT0-P0 discretization.
This type of discretization is appealing, as it leads to a minimally-sized system of equations, yet lends itself to robust linear solvers, independent of discretization and physical parameters. While we do not consider the perturbation in this paper, the main goal here is to extend this bubble-enriched discretization to make it amenable to efficient solvers, such as the  monolithic multigrid solvers described below.

After discretization, large linear systems of equations must be solved to compute the finite-element approximation to the solution of the poroelasticity equations.  This requires development of specialized preconditioners, and both block preconditioning and monolithic multigrid methods have been successfully
applied, especially for Biot's model.
For instance, robust block preconditioners are studied for the two-field formulation in \cite{CASTELLETTO2016894,NME:NME2702,adler2018robust,Castelletto2015,WHITE201655}, and for the three-field formulation in \cite{hong2018parameter,adler2019robust,FERRONATO2019108887,CASTELLETTO2016894}.
A multigrid method using alternating line Gauss-Seidel relaxation for the three-dimensional
Biot poroelasticity system is presented in \cite{naumovich2008multigrid}, which focuses on
the study of the grid-transfer operators in the multigrid method. For the quasi-static Biot model, point-wise
and line-wise box Gauss-Seidel relaxation are investigated in \cite{rodrigo2010geometric},
where local Fourier analysis (LFA) is used to help analyze and predict performance of the algorithms.
In \cite{NLA:NLA2074}, an Uzawa relaxation is employed and analyzed using LFA.
The fixed-stress split method
is used as a relaxation scheme for the two-field formulation of Biot's consolidation model in \cite{gaspar2017fixed},
where again LFA is applied to study the convergence of the multigrid method.
Similarly, a new version of the fixed-stress splitting method \cite{borregales2019partially}
is proposed for solving coupled flow and geomechanics in porous media,
modeled by a two-field formulation of Biot's equations.
Finally, multigrid waveform relaxation based on a point-wise Vanka relaxation
method is proposed for solving a collocated finite-difference discretization
of the linear Biot model in \cite{franco2018multigrid}.

Despite the work mentioned above, applications of monolithic multigrid for the discretized systems of Biot's model are rare.
In particular, given the scalable preconditioning results shown in \cite{adler2019robust} for the discretization from \cite{rodrigo2018new}, a natural question to ask is whether monolithic multigrid can compete with efficient block preconditioners.  As discussed below, initial work for this paper focused on the extension of typical monolithic multigrid relaxation schemes, known as Braess-Sarazin \cite{BRAESS1997} and Vanka \cite{VANKA1986} relaxation, to the three-field discretization from \cite{rodrigo2018new}.  While direct extensions of these methods lead to efficient preconditioners for some physical parameters, we found that they did not extend effectively to the limit of an incompressible material.  Following \cite{Schoberl1998, Schoberl1999a, Schoberl1999b}, we recognize this as an inherent consequence of the fact that the bubble-enriched P1 space does not admit a local basis for the space of divergence-free functions and, as such, standard multigrid approaches for the elasticity block are not parameter-robust.  To overcome this difficulty, we modify the discretization from \cite{rodrigo2018new} to make use of the reduced quadrature approach \cite{2013BoffiD_BrezziF_FortinM-aa, Schoberl1998, Schoberl1999a, Schoberl1999b}, which replaces exact integration of the divergence terms with that of an $L^2$ projection. Such a modification has been adopted for poroelasticity problems in~\cite{yi2017study} in order to handle locking issues when $\lambda \rightarrow \infty$. Here, we find that it also provides a ``solver-friendly'' discretization.  One of our contributions in this work is to show that using the reduced-quadrature approach still results in a well-posed discretization, which is parameter-robust, and does not lose accuracy in comparison to the discretization of \cite{rodrigo2018new}.  

Having constructed the reduced quadrature discretization, the remainder of this paper focuses on the development and analysis of optimal monolithic multigrid preconditioners for it.  In particular, we apply LFA \cite{MR1807961,wienands2004practical} to the components of the multigrid method in order to optimize parameters within the commonly used Braess-Sarazin and Vanka relaxation schemes.  In recent years, LFA has been widely used for this purpose in many contexts; for systems of PDEs, such as we consider here, it has been applied to discretizations of the Stokes equations \cite{NLA2147, HMFEMStokes, MR2840198, MR3488076, MR3217219} and, in a more limited manner, to discretizations of poroelasticity \cite{NLA:NLA2074}.  Numerical results confirm the accuracy of the LFA predictions.

In what follows, we address how the incompressibility constraint associated with the elasticity block of the coupled system affects the convergence of our proposed multigrid algorithm.  In particular, we show that the ideas of reduced-quadrature discretization and divergence-free interpolation, originally proposed and analyzed for the incompressible elasticity subproblem, can be extended to the fully-coupled Biot model.  We show that the modified discretization remains well-posed, and that we are able to develop a robust monolithic multigrid approach for the resulting three-field formulation.  Specifically, this paper is organized as follows.
In  \Cref{sec:discretization}, we introduce the stabilized finite-element discretization provided in \cite{rodrigo2018new}
for the three-field formulation of Biot's model, as well as the reduced-quadrature discretization, for which proofs of well-posedness and error estimates are given.
In \Cref{sec:mmg}, we review monolithic multigrid, with focus on both the choice of relaxation scheme for solving the discretized system
and the use of divergence-preserving interpolation operators to achieve robustness in the nearly incompressible case.
LFA for this discretization is considered in  \Cref{sec:LFA}. In  \Cref{sec:Numrical}, numerical results are presented to show the efficiency of the proposed solvers, and comparisons are given between existing block preconditioning approaches and the monolithic multigrid methods proposed here.  Finally, conclusions and remarks are drawn in  \Cref{sec:Summary}.

\section{Biot's Three-Field Formulation and its Discretization}\label{sec:discretization}

The mathematical model of the \emph{three-field formulation} of the consolidation process is described by the following system
of PDEs in a domain $\Omega\subset \mathbb{R}^{d}, d = 2, 3,$ with sufficiently smooth
boundary, $\Gamma=\partial\Omega$:
\begin{align}
	-{\rm div}(2\mu\varepsilon(\bm{u})) - \lambda\nabla({\rm div}{\bm u}) +\alpha \nabla p &=\rho \bm{g},\label{Biot-eq1}\\
	\bm{K}^{-1} \mu_f \bm{w} + \nabla p&= \rho_f\bm{g},\label{Biot-eq2}\\
	\frac{\partial}{\partial t} \left(\frac{1}{M}p+\alpha {\rm div}\bm{u}\right) +{\rm div}\bm{w}&=f.\label{Biot-eq3}
\end{align}
Here, $\mu_f$ is the viscosity of the fluid, $M$ is the Biot modulus, $\rho$ and $\rho_f$ are the bulk density and fluid density, respectively, and $\alpha = 1-\frac{K_b}{K_s}$ is the Biot-Willis
constant, with $K_b$ and $K_s$ denoting the drained and the solid-phase bulk moduli, respectively.  The absolute permeability tensor is given by $\bm{K}$ which is symmetric positive definite.
The strain tensor is denoted by $\varepsilon(\bm{u}) = \frac{1}{2}(\nabla\bm{u}+\nabla\bm{u}^{\top})$.  The unknown functions are the displacement vector $\bm{u}$, the pore pressure $p$, and the percolation velocity
of the fluid, or Darcy velocity, relative to the soil, $\bm{w}$. The vector-valued function
$\bm{g}$ represents the gravitational force.
Finally, $\displaystyle \mu = \frac{E}{2+2\nu}$ and $\displaystyle \lambda = \frac{E\nu}{(1-2\nu)(1+\nu)}$ are the Lam\'{e} coefficients where $\nu$ is the Poisson ratio and $E$ is Young's modulus.
As $\nu \rightarrow 0.5$, we have $\lambda \rightarrow \infty$, the incompressible limit that causes difficulties in numerical simulations.  Other limits that cause numerical difficulties are when the permeability, $\bm{K}\rightarrow \bm{0}$, and \cref{Biot-eq2} is dominated by its first term or, when discretized, the timestep goes to zero and \cref{Biot-eq3} is dominated by the term from timestepping. Finally, this system is subject to boundary conditions of various forms.  One typical example is:
\begin{eqnarray*}
	p &=& 0, \quad {\rm for} \,\, x\in \bar{\Gamma}_t,\quad 2\mu\varepsilon(\bm{ u})\bm{n} +\lambda {\rm div}(\bm{u})\bm{n}=\bm{0},\quad {\rm for} \,\, x\in \Gamma_t, \\
	\bm{u} &=& \bm{0},\quad {\rm for} \,\, x\in \bar{\Gamma}_c,\quad \frac{\partial p}{\partial \bm{n}}= 0,\quad {\rm for} \,\, x\in \Gamma_c,
\end{eqnarray*}
where $\bm{n}$ is the outward unit normal to the boundary, $\bar{\Gamma} = \bar{\Gamma}_t\bigcup \bar{\Gamma}_c$, with $\Gamma_t$ and $\Gamma_c$ being open (with respect to $\Gamma$) subsets of $\Gamma$ with nonzero measure.  Appropriate initial conditions for the pressure and displacement (more precisely, for ${\rm div}~\vec{u}$) are also needed.

\subsection{Finite-Element Discretization}
Following \cite{rodrigo2018new}, we consider a variational problem such that for each $t\in (0,T]$, $(\bm u(t), p(t), \bm w(t))\in \bm V \times Q \times \bm W$, with
\begin{eqnarray*}
	&&{\bm V} = \{{\bm u}\in {\bm H}^1(\Omega) \ |  \ {\bm u}|_{\overline{\Gamma}_c} = {\bm 0} \},
	\quad \quad \quad \quad Q = L^2(\Omega),\\
	&&{\bm W} = \{{\bm w} \in \bm{H}(\dive,\Omega) \ | \  ({\bm w}\cdot {\bm n})|_{\Gamma_c} = 0\},
\end{eqnarray*}
where ${\bm H}^1(\Omega)$ is the space of square integrable
vector-valued functions whose first derivatives are also square
integrable, and $\bm{H}(\dive,\Omega)$ contains the square integrable
vector-valued functions with square integrable divergence.

Using backward Euler as a time discretization on a
time interval $(0,T]$ with constant time-step size
$\tau$, the discrete variational form for Biot's three-field consolidation model, \cref{Biot-eq1}-\cref{Biot-eq3}, is written as:  Find $(\bm u_h^m, p_h^m, \bm w_h^m)\in \bm V_h \times Q_h \times \bm W_h$
such that
\begin{align}
	a(\bm{u}_h^m,\bm{v}_h) - (\alpha p_h^m,\dive \bm{v}_h) & =
	(\rho\bm g,\bm{v}_h), \quad \forall \  \bm{v}_h\in \bm V_h, \label{discrete1}\\
	\tau({\bm K}^{-1}\mu_f\bm{w}_h^m,\bm{r}_h) - \tau(p_h^m,\dive \bm{r}_h) &=
	\tau(\rho_f {\bm g}, \bm{r}_h), \quad \forall \ \bm{r}_h\in \bm W_h,\label{discrete2}\\
	-\left(\frac{1}{M} p^m_h ,q_h \right)-\left(\alpha\dive \bm{u}_h^m,q_h\right) - \tau(\dive \bm{w}_h^m,q_h)&=
	-(\hat{f}, q_h), \quad \forall \ q_h \in Q_h,\label{discrete3}
\end{align}
where $(\cdot,\cdot)$ denotes the standard $L^2(\Omega)$ inner product.
Here, $(\bm{u}_h^m, p_h^m, \bm{w}_h^m) $
is an approximation to
$\left(\bm{u}(\cdot, t_m), p(\cdot, t_m), \bm{w}(\cdot, t_m)\right), $
at time
$t_m = m\tau, \ m = 1,2,\ldots$, $(\hat{f}, q_h) =\tau(f,q_h) + \left(\frac{1}{M} p^{m-1}_h ,q_h \right)+\left(\alpha\dive \bm{u}_h^{m-1},q_h\right)$ and
$a(\bm{u},\bm{v}) = 2\mu\left ({ \varepsilon}(\bm{u}),{ \varepsilon}(\bm{v})\right ) +
\lambda\left ( \dive\bm{u},\dive\bm{v}\right )
$
is the usual weak form for linear elasticity.
Note that \cref{discrete2} has been scaled by $\tau$ and \cref{discrete3} has been scaled by $-1$ to make the system symmetric.

For finite-element spaces, we consider linear elements (P1), enriched with bubble functions on faces for $\bm{V}_h\subset \bm{V}$.
These face-normal bubble functions are quadratic in 2D and cubic in 3D.  Their degrees of freedom are defined as the integrated normal displacement across the associated faces.  This space is covered in-depth in Chapter 2.1 of \cite{GR1986}.
We choose $Q_h\subset Q$ as the piecewise constant space (P0) for the pressure, and $\bm{W}_h\subset\bm{W}$ as the standard lowest-order Raviart-Thomas space (RT0)
for the Darcy velocity.
It has been shown that this discretization is a stable finite-element approximation, see \cite{rodrigo2018new}.

Finally, this discrete variational form can be represented in block matrix form as
\begin{equation}\label{block_form}
	\mathcal{ A} \left(
	\begin{array}{c}
		{\bm u} \\
		{\bm w} \\
		p
	\end{array}
	\right) =
	{\bm b}, \ \ \hbox{with} \ \
	\mathcal{ A} = \left(
	\begin{array}{ccc}
		A_{\bm{u}}  & 0 & \alpha B_{\bm{u}}^{\top}\\
		0  & \tau M_{\bm w}& \tau B_{\bm w}^{\top} \\
		\alpha B_{\bm{u}} & \tau B_{\bm w} & -\frac{1}{M} M_p
	\end{array}
	\right).
\end{equation}
The blocks in the matrix $\mathcal{A}$ correspond to the following
bilinear forms:
\begin{align*}
	a(\bm{u}_h,\bm{v}_h) &\rightarrow A_{\bm u},   &\quad
	-( \dive \bm{u}_h, q_h)   &\rightarrow B_{\bm u}&\quad -(\dive \bm{w}_h,q_h) &\rightarrow B_{\bm w},\\
	({\bm K}^{-1}\mu_f\bm{w}_h,\bm{r}_h) &\rightarrow  M_{\bm w}, &\quad \left( p_h ,q_h \right)     &\rightarrow M_p. &&
\end{align*}

\subsection{Solver Incompatibility}\label{sec:solve1}
While the above discretization is well-posed and, as shown in \cite{rodrigo2018new}, is robust to variations in the physical and discretization parameters,
solving the resulting linear system in a similarly parameter-robust manner is not straightforward.  A block-preconditioning framework was proposed in \cite{adler2019robust} for the solution of the linear system and the proposed approaches were proven to be parameter-robust under the assumption that each diagonal block of the preconditioner can be solved in a parameter-robust manner.   While \cite{adler2019robust} contains a detailed parameter study, the primary measure of convergence there was in outer iterations of FGMRES, where the inner iterations (to approximate solves with the diagonal blocks of the block preconditioners) were done to fixed tolerances with AMG-preconditioned GMRES.  As numerical results presented below in \Cref{sec:Numrical} will show, while the outer iterations reported in \cite{adler2019robust} are robust to the physical parameters (in particular, the incompressible limit), the inner iterations are not.

In preliminary investigations for this paper, similar behavior was seen for the monolithic multigrid methods detailed below.  There are several common relaxation schemes considered when applying monolithic multigrid to block-structured saddle-point problems, such as the system in \cref{block_form}, which will be described in more detail below.  Braess-Sarazin approaches use approximations to the block factorization of $\mathcal{A}$ as relaxation schemes.  \textit{Exact} Braess-Sarazin relaxation (BSR) is based on exact solution of the approximate Schur complement(s) in such a factorization, while \textit{inexact} Braess-Sarazin methods also introduce an approximation to the Schur complement(s).  An alternative approach is to use Vanka relaxation schemes (see \Cref{subsec:Vanka}), which are block overlapping Schwarz methods, with small blocks chosen to reflect the saddle-point structure of the system.  \cref{tab:naive} shows that, while exact Braess-Sarazin relaxation is effective in a parameter-independent manner, convergence suffers for both inexact Braess-Sarazin and Vanka relaxation schemes.

\begin{table}[htp]
	\footnotesize
	\begin{center}
		\caption{Measured convergence factors for monolithic multigrid applied to \cref{block_form}, with $\bm{K} = k \bm{I}$, $k=10^{-6}$ and varying $\nu$, for a uniform mesh $h=1/64$ of the unit square and a time step size of $\tau=1.0$.}
		\label{tab:naive}
		\begin{tabular}{|l|l l l l l l|}\hline
			\backslashbox{}{$\nu$} & $0.0$ & $0.2$ & $0.4$ & $0.45$ & $0.49$ & $0.499$\\\hline
			Exact BSR          & 0.067 & 0.067 & 0.067 &  0.067 & 0.067  & 0.067  \\
			Inexact BSR          & 0.440 & 0.471 & 0.586 &  0.659 & 0.790  & 0.968  \\
			Vanka & 0.515 & 0.513 & 0.589 &  0.659 & 0.794  & 0.970  \\\hline
		\end{tabular}
	\end{center}
\end{table}

The degradation in performance from exact to inexact Braess-Sarazin relaxation  as $\nu \rightarrow 0.5$ in \cref{tab:naive} was carefully studied.  For both Braess-Sarazin variants considered, we took a Schur complement onto the displacement degrees of freedom, and invested significant effort into constructing relaxation schemes for that Schur complement that would lead to a robust inexact Braess-Sarazin variant.
The primary source of the problem became clear when looking at the dominant errors in the displacements after running two-grid cycles with either the inexact Braess-Sarazin or Vanka relaxation, visualized for the Vanka case in \cref{fig:globalDivFree}.  In essence, this error reflects a \textit{globally-supported} divergence-free null-space that is difficult to eliminate using local relaxation schemes.  As we next show, this arises from the exact evaluation of the $(\ddiv\bm{u}, \ddiv\bm{v})$ term within the discretization, resulting in a discretization
that is inherently not ``solver-friendly'', due to the lack of a local basis for the space of (nearly) divergence-free functions.  To address this, we modify the discretization using a reduced quadrature approach \cite{malkus1978mixed,2013BoffiD_BrezziF_FortinM-aa}, as suggested in~\cite{yi2017study} for poroelasticity problems.

\begin{figure}[htp]
	\centering
	\includegraphics[scale=0.18]{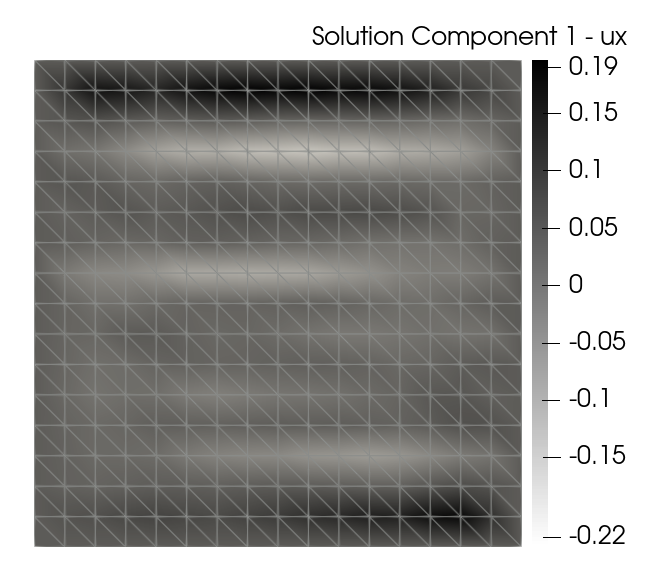}
	\includegraphics[scale=0.18]{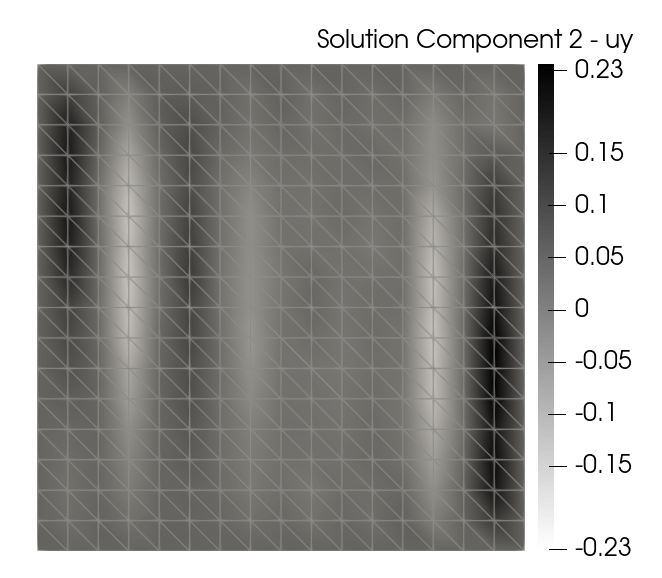}
	\includegraphics[scale=0.18]{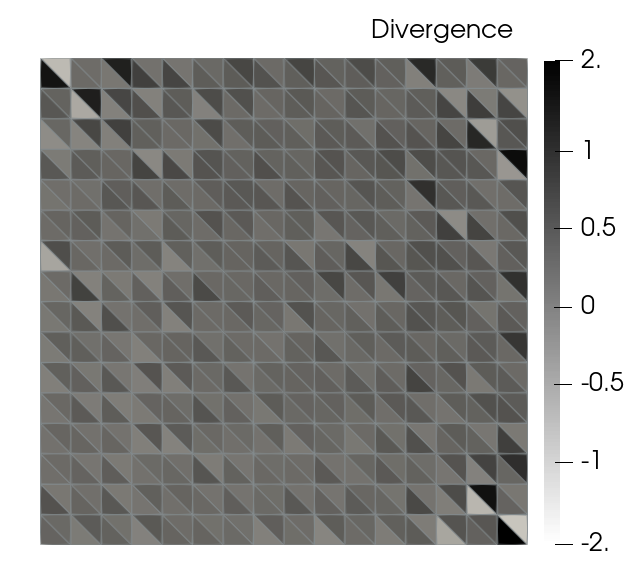}
	\caption{Error and divergence of the error for displacement after 40 cycles of two-level multigrid with Vanka relaxation, for $\nu=0.49$, applied to a problem with zero right-hand side and random initial guess. The divergence of the error illustrates neighboring element pairs with divergence of similar magnitude but opposite sign, indicating a globally-supported divergence-free null-space.} \label{fig:globalDivFree}
\end{figure}

\subsection{Reduced Quadrature}\label{sec:rq}
As recognized in \cite{Schoberl1998, Schoberl1999a, Schoberl1999b}, the non-local nature of the basis for the divergence-free spaces arises from the direct evaluation of the $(\ddiv\bm{u}, \ddiv\bm{v})$ term in the weak form, since the discrete divergence of the displacement space is not a subset of the piecewise constant pressure space.
To avoid this, we implement a reduced integration approach\cite{malkus1978mixed,2013BoffiD_BrezziF_FortinM-aa,yi2017study} and replace $(\ddiv\bm{u}, \ddiv\bm{v})$ with $(P_{Q_h} \ddiv\bm{u}, P_{Q_h}\ddiv\bm{v})$, where $P_{Q_h}$ is the $L^2$-projection from $Q$ onto $Q_h$, the space of piecewise constant functions. With this reduced integration approach, a basis for the space of divergence-free functions is readily constructed with local support, allowing local relaxation schemes to be effective for divergence-free components.

To illustrate this further, consider that the discretization for displacements has a total of $2 N_v + N_e$ degrees of freedom (DoFs),
where $N_v$ is the number of vertices in the mesh, and $N_e$ is the number of edges.
By direct computation, around each vertex in the mesh, we can introduce a local basis of three divergence-free functions, shown in \cref{fig:DivFreeBasis}, resulting in $3 N_v$ divergence-free basis functions.
The reduced quadrature approach constrains $\ddiv \bm{u}$
to be in the piecewise constant pressure space,
thus, there are $N_T - 1$ divergence-free constraints,
where $N_T$ is the number of triangular elements.
Now, subtracting the number of divergence-free constraints from the total DoFs,
$(2N_v+N_e)-(N_T-1) = 3N_v$,
we get the number of divergence-free basis functions.
Thus, the reduced quadrature approach fully supports the divergence-free functions
through the local basis functions in \cref{fig:DivFreeBasis}.
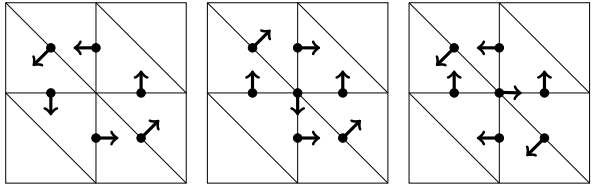
\begin{figure}[h!]
	\begin{center}
		\scalebox{0.8}{
			\begin{tikzpicture}
				\def\h{1.5}
				\node (a) at (\h*0,\h*2) {}; \node (b) at (\h*1,\h*2) {}; \node (bb) at (\h*2,\h*2) {};
				\node (c) at (\h*0,\h*1) {}; \node (d) at (\h*1,\h*1) {}; \node (e) at (\h*2,\h*1) {};
				\node (ff) at (\h*0,\h*0) {}; \node (f) at (\h*1,\h*0) {}; \node (g) at (\h*2,\h*0) {};
				\draw (a.center) -- (b.center) -- (e.center) -- (g.center) -- (f.center) -- (c.center) -- (a.center);
				\draw (a.center) -- (d.center) -- (g.center);
				\draw (c.center) -- (d.center) -- (e.center);
				\draw (b.center) -- (d.center) -- (f.center);
				\draw (c.center) -- (ff.center) -- (f.center);
				\draw (b.center) -- (bb.center) -- (e.center);
				\filldraw ($(d)+\h*(0,0.5)$)    circle(2pt);
				\draw[line width=1.5pt,->] ($(d)+\h*(0,0.5)$) -- ($(d)+\h*(0,0.5)+\h*(-0.25,0)$);
				\filldraw ($(d)+\h*(0.5,0)$)    circle(2pt);
				\draw[line width=1.5pt,->] ($(d)+\h*(0.5,0)$) -- ($(d)+\h*(0.5,0)+\h*(0,0.25)$);
				\filldraw ($(d)-\h*(0,0.5)$)    circle(2pt);
				\draw[line width=1.5pt,->] ($(d)-\h*(0,0.5)$) -- ($(d)-\h*(0,0.5)+\h*(0.25,0)$);
				\filldraw ($(d)-\h*(0.5,0)$)    circle(2pt);
				\draw[line width=1.5pt,->] ($(d)-\h*(0.5,0)$) -- ($(d)-\h*(0.5,0)+\h*(0,-0.25)$);
				\filldraw ($(d)+\h*(0.5,-0.5)$) circle(2pt);
				\draw[line width=1.5pt,->] ($(d)+\h*(0.5,-0.5)$) -- ($(d)+\h*(0.5,-0.5)+\h*(0.2,0.2)$);
				\filldraw ($(d)+\h*(-0.5,0.5)$) circle(2pt);
				\draw[line width=1.5pt,->] ($(d)+\h*(-0.5,0.5)$) -- ($(d)+\h*(-0.5,0.5)+\h*(-0.2,-0.2)$);
			\end{tikzpicture}
			\begin{tikzpicture}
				\def\h{1.5}
				\node (a) at (\h*0,\h*2) {}; \node (b) at (\h*1,\h*2) {}; \node (bb) at (\h*2,\h*2) {};
				\node (c) at (\h*0,\h*1) {}; \node (d) at (\h*1,\h*1) {}; \node (e) at (\h*2,\h*1) {};
				\node (ff) at (\h*0,\h*0) {}; \node (f) at (\h*1,\h*0) {}; \node (g) at (\h*2,\h*0) {};
				\draw (a.center) -- (b.center) -- (e.center) -- (g.center) -- (f.center) -- (c.center) -- (a.center);
				\draw (a.center) -- (d.center) -- (g.center);
				\draw (c.center) -- (d.center) -- (e.center);
				\draw (b.center) -- (d.center) -- (f.center);
				\draw (c.center) -- (ff.center) -- (f.center);
				\draw (b.center) -- (bb.center) -- (e.center);
				\filldraw ($(d)+\h*(0,0.5)$)    circle(2pt);
				\draw[line width=1.5pt,->] ($(d)+\h*(0,0.5)$) -- ($(d)+\h*(0,0.5)+\h*(0.25,0)$);
				\filldraw ($(d)+\h*(0.5,0)$)    circle(2pt);
				\draw[line width=1.5pt,->] ($(d)+\h*(0.5,0)$) -- ($(d)+\h*(0.5,0)+\h*(0,0.25)$);
				\filldraw ($(d)-\h*(0,0.5)$)    circle(2pt);
				\draw[line width=1.5pt,->] ($(d)-\h*(0,0.5)$) -- ($(d)-\h*(0,0.5)+\h*(0.25,0)$);
				\filldraw ($(d)-\h*(0.5,0)$)    circle(2pt);
				\draw[line width=1.5pt,->] ($(d)-\h*(0.5,0)$) -- ($(d)-\h*(0.5,0)+\h*(0,0.25)$);
				\filldraw ($(d)+\h*(0.5,-0.5)$) circle(2pt);
				\draw[line width=1.5pt,->] ($(d)+\h*(0.5,-0.5)$) -- ($(d)+\h*(0.5,-0.5)+\h*(0.2,0.2)$);
				\filldraw ($(d)+\h*(-0.5,0.5)$) circle(2pt);
				\draw[line width=1.5pt,->] ($(d)+\h*(-0.5,0.5)$) -- ($(d)+\h*(-0.5,0.5)+\h*(0.2,0.2)$);
				\filldraw (d) circle (2pt);
				\draw[line width=1.5pt,->] ($(d)+(0,0.01)$) -- ($(d)+\h*(0,-0.25)$);
			\end{tikzpicture}
			\begin{tikzpicture}
				\def\h{1.5}
				\node (a) at (\h*0,\h*2) {}; \node (b) at (\h*1,\h*2) {}; \node (bb) at (\h*2,\h*2) {};
				\node (c) at (\h*0,\h*1) {}; \node (d) at (\h*1,\h*1) {}; \node (e) at (\h*2,\h*1) {};
				\node (ff) at (\h*0,\h*0) {}; \node (f) at (\h*1,\h*0) {}; \node (g) at (\h*2,\h*0) {};
				\draw (a.center) -- (b.center) -- (e.center) -- (g.center) -- (f.center) -- (c.center) -- (a.center);
				\draw (a.center) -- (d.center) -- (g.center);
				\draw (c.center) -- (d.center) -- (e.center);
				\draw (b.center) -- (d.center) -- (f.center);
				\draw (c.center) -- (ff.center) -- (f.center);
				\draw (b.center) -- (bb.center) -- (e.center);
				\filldraw ($(d)+\h*(0,0.5)$)    circle(2pt);
				\draw[line width=1.5pt,->] ($(d)+\h*(0,0.5)$) -- ($(d)+\h*(0,0.5)+\h*(-0.25,0)$);
				\filldraw ($(d)+\h*(0.5,0)$)    circle(2pt);
				\draw[line width=1.5pt,->] ($(d)+\h*(0.5,0)$) -- ($(d)+\h*(0.5,0)+\h*(0,0.25)$);
				\filldraw ($(d)-\h*(0,0.5)$)    circle(2pt);
				\draw[line width=1.5pt,->] ($(d)-\h*(0,0.5)$) -- ($(d)-\h*(0,0.5)+\h*(-0.25,0)$);
				\filldraw ($(d)-\h*(0.5,0)$)    circle(2pt);
				\draw[line width=1.5pt,->] ($(d)-\h*(0.5,0)$) -- ($(d)-\h*(0.5,0)+\h*(0,0.25)$);
				\filldraw ($(d)+\h*(0.5,-0.5)$) circle(2pt);
				\draw[line width=1.5pt,->] ($(d)+\h*(0.5,-0.5)$) -- ($(d)+\h*(0.5,-0.5)+\h*(-0.2,-0.2)$);
				\filldraw ($(d)+\h*(-0.5,0.5)$) circle(2pt);
				\draw[line width=1.5pt,->] ($(d)+\h*(-0.5,0.5)$) -- ($(d)+\h*(-0.5,0.5)+\h*(-0.2,-0.2)$);
				\filldraw (d) circle (2pt);
				\draw[line width=1.5pt,->] ($(d)+(0,0.01)$) -- ($(d)+\h*(0.25,0)$);
			\end{tikzpicture}
		}
	\end{center}
	\caption{The local divergence-free bases supported by the discretization.} \label{fig:DivFreeBasis}
\end{figure}

Therefore, we define the bilinear form for the reduced quadrature discretization as
\begin{equation*}
	a^{\text{RQ}}(\bm{u},\bm{v}) := 2\mu \left ({ \varepsilon}(\bm{u}),{ \varepsilon}(\bm{v})\right ) + \lambda(P_{Q_h} \ddiv\bm{u}, P_{Q_h}\ddiv\bm{v}).
\end{equation*}
Using this, the poroelastic system is then written as
\begin{equation} \label{block_form_rq}
	\mathcal{A}^{\text{RQ}} = \left(
	\begin{array}{ccc}
		A_{\bm{u}}^{\text{RQ}} &0                      & \alpha B_{\bm{u}}^{\top} \\
		0                      & \tau M_{\bm w}       & \tau B_{\bm w}^{\top} \\
		\alpha B_{\bm u}      & \tau B_{\bm w}       & -\frac{1}{M} M_p
	\end{array}
	\right),
\end{equation}
where $a^{\text{RQ}}(\bm{u}_h,\bm{v}_h)\rightarrow A_{\bm{u}}^{\text{RQ}}$.
We next show that this reduced quadrature approach remains
well-posed independent of the physical and discretization parameters.
To do this, we first introduce the following lemma concerning the Stokes inf-sup condition:

\begin{lemma} \label{lem:stokesRQ}
	Let the pair of finite-element spaces $\bm{V}_h \times Q_h$ be Stokes-stable, i.e., satisfy the inf-sup condition~\cite{GR1986},
	\begin{equation*}
		\sup_{\bm{v}\in\bm{V}_h} \frac{(\ddiv\bm{v},p)}{\|\bm{v}\|_1} \geq \gamma_B^0 \|p\|, \quad \forall \ p\in Q_h,
	\end{equation*}
	where $\gamma^0_B > 0$ is a constant that does not depend on mesh size.  Then, for any $p\in Q_h$
	\begin{equation}\label{eqn:stokesA-inf-sup-RQ}
		\sup_{\bm{v}\in\bm{V_h}} \frac{( \ddiv \bm{v},p)}{\|\bm{v}\|_{A_{\bm{u}}^{\rm{RQ}}}}
		\geq \frac{\gamma_B^0}{\sqrt{d}\zeta}\|p\| =: \frac{\gamma_B}{\zeta}\|p\|,
	\end{equation}
	where $\| \bm{v} \|_{A^{\rm{RQ}}_{\bm{u}}}^2 := a^{\rm{RQ}}(\bm{v},\bm{v})$, $d$ is the dimension,
	and $\zeta := \sqrt{\lambda + 2\mu/d}$.
\end{lemma}
\begin{proof}
	Using the properties of projection operators, we have that$ \| P_{Q_h} \ddiv\bm{v} \| \leq \| \ddiv\bm{v} \|$ for all $\bm{v}\in\bm{V}_h$.  This, along with the definitions of $ A_{\bm u}^{\text{RQ}} $ and $A_{\bm u}$, yields
	\begin{equation}\label{eqn:RQ_leq_A}
		\| \bm{v} \|_{A_{\bm u}^{\text{RQ}}} \leq \| \bm{v} \|_{A_{\bm u}} \text{ for all }\bm{v}\in\bm{V}_h.
	\end{equation}
	
	Next, by direct computation and applying Young's inequality, we have that $(\ddiv \bm{v}, \ddiv \bm{v}) \leq d (\varepsilon(\bm{v}),\varepsilon(\bm{v}))$.
	This implies that $a(\bm{v},\bm{v}) \leq (2\mu + d\lambda)(\varepsilon(\bm{v}),\varepsilon(\bm{v}))$,
	and, through another application of Young's inequality,
	we have $\|\bm{v}\|_{A_{\bm{u}}} \leq \sqrt{d}\zeta\|\bm{v}\|_1$, with
	$a(\bm{v},\bm{v}) =: \|\bm{v}\|^2_{A_{\bm{u}}}$.
	Then, for any $p\in Q_h$,
	\begin{equation}\label{eqn:stokesA-inf-sup-v2}
		\sup_{\bm{v}\in\bm{V_h}} \frac{(B_{\bm{u}}\bm{v},p)}{\|\bm{v}\|_{A_{\bm{u}}}}
		\geq \frac{\gamma_B^0}{\sqrt{d}\zeta}\|p\| =: \frac{\gamma_B}{\zeta}\|p\|.
	\end{equation}
	Thus, \cref{eqn:stokesA-inf-sup-v2} and \cref{eqn:RQ_leq_A} give \cref{eqn:stokesA-inf-sup-RQ}.
\end{proof}

Note that, since the norm $\|\cdot\|_{A^{\rm{RQ}}_{\bm{u}}}$ is parameter-dependent,
in the large $\lambda$ limit, both sides of \cref{eqn:stokesA-inf-sup-RQ} behave as $1/\sqrt{\lambda}$.
We now show that the reduced-quadrature discretization is well-posed,
using the fact that the bubble-enriched P1-RT0-P0 discretization is
\textit{Stokes-Biot stable} (see Definition 3.1 in \cite{rodrigo2018new}).

\begin{theorem}\label{thm:rq_wellposed}
	Let $\bm{X}_h = (\bm{V}_h, \bm{W}_h, Q_h)$ be Stokes-Biot stable, that is,
	\begin{itemize}
		\item $\exists \ C_{\bm{V}}>0$ such that  $a(\bm{u}, \bm{v})  \leq C_{\bm{V}} \| \bm{u} \|_1 \| \bm{v} \|_1$, for all $\bm{u}, \bm{v} \in {\bm{V}}_h$;
		\item $\exists \ \alpha_{\bm{V}} > 0$ such that $a(\bm{u}, \bm{u})  \geq \alpha_{\bm{V}} \| \bm{u} \|_1^2$, for all $\bm{u}\in {\bm{V}}_h$;
		\item $({\bm{W}}_h,{Q}_h)$ is
		Poisson stable, satisfying the necessary stability and continuity
		conditions for the mixed formulation of Poisson's
		equation; and
		\item The pair of spaces $({\bm{V}}_h,{Q}_h)$ is Stokes stable.
	\end{itemize}
	For $\bm{x} = (\bm{u},\bm{w},p)\in\bm{X}_h$ and $\bm{y} = (\bm{v},\bm{w},p)\in\bm{X}_h$, define
	\begin{align}
		\mathcal{B}(\bm{x},\bm{y}) =&
		a^{\rm{RQ}}(\bm{u},\bm{v}) - \left(\alpha p, \ddiv \bm{v}\right)
		+ \tau({\bm K}^{-1}\mu_f\bm{w},\bm{r}) - \tau(p,\ddiv \bm{r}) \label{eqn:bilinearform}\\
		&- \tau(\ddiv \bm{w},q) - \left(\frac{1}{M} p ,q \right) - \left(\alpha \ddiv \bm{u}, q\right), \notag \\
		\| \bm{x} \|_{\mathcal{D}^{\rm{RQ}}}^2 = &
		\| \bm{u} \|_{A_{\bm u}^{\rm{RQ}}}^2 +
		c_p^{-1} \| p \|^2 +
		\tau \| \bm{w} \|_{M_{\bm w}}^2 + \tau^2 c_p \| \ddiv \bm{w} \|^2,\label{eqn:weighted_norm_matrix}
	\end{align}
	where $\|\bm{w}\|_{M_{\bm w}}^2 := (\bm{K}^{-1}\mu_f \bm{w}, \bm{w})$, and $c_p = \left(\frac{\alpha^2}{\zeta^2} + \frac{1}{M}\right)^{-1}$.  Then
	\begin{align}
		& \sup_{\bm{0} \neq \bm{x} \in \bm{X}_h} \sup_{\bm{0} \neq \bm{y} \in \bm{X}_h }
		\frac{ \mathcal{B}( \bm{x}, \bm{y})}{ \|\bm{x}\|_{\mathcal{D}^{\rm{RQ}}} \|\bm{y}\|_{\mathcal{D}^{\rm{RQ}}}}
		\leq \tilde{\varsigma}, \label{supsup-AD-rq} \\
		& \Inf_{\bm{0} \neq \bm{y} \in \bm{X}_h} \sup_{\bm{0} \neq \bm{x} \in \bm{X}_h}
		\frac{ \mathcal{B}( \bm{x}, \bm{y})}{ \|\bm{x}\|_{\mathcal{D}^{\rm{RQ}}} \|\bm{y}\|_{\mathcal{D}^{\rm{RQ}}} }
		\geq \tilde{\gamma}, \label{infsup-AD-rq}
	\end{align}
	where the constants $\tilde{\varsigma}$ and $\tilde{\gamma}$ are independent
	of the physical and discretization parameters.
\end{theorem}
\begin{proof}
	Using \cref{lem:stokesRQ}, we know that
	for a given $p \in Q_h$, there exists $\bm{z} \in \bm{V}_h$, such that
	$(p, \ddiv \bm{z}) \geq \frac{\gamma_B}{\zeta} \| p \|^2$
	and $\| \bm{z} \|_{A_{\bm{u}}^{\text{RQ}}} = \| p \|$.
	Let $\bm{v} = \bm{u} - \psi_1 \bm{z}$, $\bm{r} = \bm{w}$, and
	$q = -p - \psi_2 \tau \ddiv \bm{w}$ for constants $\psi_1$ and $\psi_2$ that will be specified later.  Then, by the Cauchy-Schwarz and Young's inequality,
	\begin{align*}
		\mathcal{B}(\bm{x}, \bm{y})
		= & \| \bm{u} \|_{A_{\bm{u}}^{\text{RQ}}}^2 - \psi_1 a^{\text{RQ}}( \bm{u}, \bm{z})
		+ \psi_1 \alpha(p,  \ddiv \bm{z}) + \tau \| \bm{w} \|^2_{M_{\bm w}}
		+ \frac{1}{M} \| p \|^2 \\ & + \psi_2 \tau \frac{1}{M}(p,  \ddiv \bm{w})
		+ \psi_2 \alpha \tau ( P_{Q_h} \ddiv \bm{u},\ddiv \bm{w})
		+ \psi_2 \tau^2 \|  \ddiv \bm{w} \|^2\\
		\geq & \| \bm{u} \|_{A_{\bm{u}}^{\text{RQ}}}^2 - \frac{1}{2} \| \bm{u} \|_{A_{\bm{u}}^{\text{RQ}}}^2
		- \frac{\psi_1^2}{2} \| \bm{z}\|_{A_{\bm{u}}^{\text{RQ}}}^2
		+ \psi_1 \frac{\alpha \gamma_B}{\zeta} \| p \|^2
		+ \tau \| \bm{w} \|^2_{M_{\bm w}} + \frac{1}{M} \| p \|^2 \\
		& - \frac{3 \psi_2}{2} \frac{1}{M^2} \| p \|^2
		- \frac{\psi_2}{6} \tau^2 \| \ddiv \bm{w} \|^2
		- \frac{\psi_2}{2} \alpha^2 \| P_{Q_h} \ddiv \bm{u} \|^2
		- \frac{\psi_2}{2} \tau^2\| \ddiv \bm{w} \|^2\\
		& + \psi_2 \tau^2 \| \ddiv \bm{w} \|^2.
	\end{align*}
	As in the proof of \Cref{lem:stokesRQ},
	\[
	\frac{1}{d}(P_{Q_h}\ddiv\bm{u},P_{Q_h}\ddiv\bm{u}) \leq
	\frac{1}{d}(\ddiv\bm{u},\ddiv\bm{u}) \leq
	(\epsilon(\bm{u}),\epsilon(\bm{u})).\]
	Then, by direct calculation and the definition of $A_{\bm u}^{\text{RQ}}$, we have
	\begin{equation}\label{eqn:B_less_Arq}
		\| P_{Q_h} \ddiv \bm{u} \| \leq \frac{1}{\zeta} \| \bm{u} \|_{A_{\bm u}^{\text{RQ}}}.
	\end{equation}
	Combining terms and applying \cref{eqn:B_less_Arq} gives
	\begin{align*}
		\mathcal{B}(\bm{x}, \bm{y})
		\geq & \left( \frac{1}{2} - \frac{\psi_2}{2 }\frac{\alpha^2}{\zeta^2} \right) \| \bm{u} \|_{A_{\bm{u}}^{\text{RQ}}}^2
		+ \tau \| \bm{w} \|^2_{M_{\bm w}} + \frac{1}{3} \psi_2 \tau^2 \|  \ddiv \bm{w} \|^2  \\
		& + \left( \psi_1 \frac{\alpha \gamma_B}{\zeta} - \frac{\psi^2_1}{2} \right) \| p \|^2
		+ \left( 1 - \frac{3}{4} \frac{2\psi_2}{M} \right) \frac{1}{M} \| p \|^2.
	\end{align*}
	Choosing $\psi_1 = \frac{\alpha \gamma_B}{2 \zeta}$
	and $\psi_2 =  \frac{1}{2}  \left( \frac{\alpha^2}{\zeta^2} + \frac{1}{M} \right)^{-1}$ then gives
	\begin{align*}
		\mathcal{B}(\bm{x}, \bm{y})
		\geq & \left(\frac{1}{2}-\frac{1}{4}\right) \| \bm{u} \|_{A_{\bm{u}}^{\text{RQ}}}^2 + \tau \| \bm{w} \|^2_{M_{\bm w}}
		+ \frac{1}{6} \tau^2 \left( \frac{\alpha^2}{\zeta^2}
		+ \frac{1}{M} \right)^{-1} \|  \ddiv \bm{w} \|^2 \\
		& + \left( \frac{3\alpha^2\gamma_B^2}{8\zeta^2} \right) \| p \|^2
		+ \left( 1 - \frac{3}{4} \right) \frac{1}{M} \| p \|^2 \\
		\geq & \bar{\gamma} \| \left( \bm{u},  \bm{w}, p\right) \|_{\mathcal{D}^{\text{RQ}}}^2,
	\end{align*}
	where $\bar{\gamma} = \min \left\{  \frac{1}{6}, \frac{3\gamma_B^2}{8}  \right \}$.
	Then, by the triangle inequality,
	\begin{align*}
		\|\bm{y}\|_{\mathcal{D}^{\text{RQ}}}^2 &=
		\| \bm{v} \|_{A_{\bm u}^{\text{RQ}}}^2 +
		\left(\frac{\alpha^2}{\zeta^2} + \frac{1}{M}\right) \| q \|^2 +
		\tau \| \bm{r} \|_{M_{\bm w}}^2 + \tau^2 c_p \| \ddiv \bm{r} \|^2 \leq (\gamma^*)^2 \| \bm{x} \|_{\mathcal{D}^{\text{RQ}}}^2,
	\end{align*}
	where $(\gamma^*)^2 = \max\left\{ 2, \frac{\gamma_B^2}{4} \right\}$.
	Thus, the bilinear form $\mathcal{B}(\cdot,\cdot)$ defined in \cref{eqn:bilinearform} satisfies
	\cref{infsup-AD-rq} with $\tilde{\gamma} = \gamma^* / \bar{\gamma}$.
	For the upper bound, \cref{supsup-AD-rq}, using Cauchy-Schwarz and \cref{eqn:B_less_Arq}, we have
	$\mathcal{B}(\bm{x},\bm{y})
	\leq 8\|\bm{x}\|_{\mathcal{D}^{\text{RQ}}} \|\bm{y}\|_{\mathcal{D}^{\text{RQ}}}$, which completes the proof.
\end{proof}

\begin{remark}
	To better understand the choice of the weighted norm \cref{eqn:weighted_norm_matrix}, consider two limiting cases.  When $\lambda \rightarrow \infty$, $\mathcal{B}(\bm{x},\bm{y})$ is dominated by\newline $\lambda (P_{Q_h} \ddiv \bm{u}, P_{Q_h} \ddiv \bm{v})$, which corresponds to the dominating term $\lambda \| P_{Q_h} \ddiv \bm{u} \|^2$ in the weighted norm. When $\tau \rightarrow 0$,  $\mathcal{B}(\bm{x}, \bm{y})$ reduces to $a^{\text{RQ}}(\bm{u,\bm{v}}) - (\alpha p, \ddiv \bm{v}) - (\alpha \ddiv \bm{u}, q) - \frac{1}{M}(p,q)$, which is a Stokes-like problem. The weighted norm \cref{eqn:weighted_norm_matrix}, in this case, reduces to $\|\bm{u}\|^2_{A_{\bm{u}}^{\text{RQ}}} + c_p^{-1} \|p\|^2$, which is a proper choice for Stokes-type problems.  Thus, the weighted norm \cref{eqn:weighted_norm_matrix} is a proper choice in those limiting cases.
\end{remark}

	\begin{remark}
		In~\cite{mardal2021accurate}, the \emph{minimal Stokes-Biot stability condition} was proposed, under which a wider class of discretizations can be shown to be parameter-robust for solving the three-field formulation~\eqref{Biot-eq1}-\eqref{Biot-eq3}.  That result also applies to the reduced-quadrature discretization presented here, and the conclusions of \cref{thm:rq_wellposed} still hold if we assume $\bm{X}_h = (\bm{V}_h, \bm{W}_h, Q_h)$ to be \emph{minimal Stokes-Biot stable}, i.e., replacing the condition that $(\bm{W}_h, Q_h)$ is Poisson stable by $\ddiv \bm{W}_h \subset Q_h$.  In fact, the proof of \cref{thm:rq_wellposed} uses only the minimal Stokes-Biot stability condition.  This means that the reduced-quadrature technique can be applied to other discretizations that are minimal Stokes-Biot stable but not Stokes-Biot stable, e.g., the bubble-enriched P1-P1-P0 and P2-P1-P0 discretizations. We refer to~\cite{mardal2021accurate} for further discussion of spaces that satisfy the minimal Stokes-Biot stability condition. 
	\end{remark}

\begin{remark}
	In~\cite{yi2017study}, it has been shown that the reduced-quadrature discretization is well-posed independent of the discretization parameters by using the traditional Brezzi theory for saddle-point systems~\cite{brezzi1974existence}.  Here, with the help of Stokes-Biot stability and properly chosen weighted norm, we show that the reduced-quadrature discretization is well-posed independent of the physical parameters as well.  This implies that the reduced quadrature approach is parameter-robust and also does not destroy the approximation properties of the bubble-enriched P1-RT0-P0 discretization~\cite{rodrigo2018new}.
\end{remark}

\section{Monolithic Multigrid}\label{sec:mmg}
Preconditioners for coupled systems, such as the reduced quadrature discretization in \cref{block_form_rq}, generally fall into two classes, those based on block-factorization approaches and those based on monolithic multigrid.  The block-factorization approach was considered for the discretization from \cite{rodrigo2018new} in \cite{adler2019robust}; here, we focus on monolithic multigrid, extending recent studies in~\cite{NLA2147, HMFEMStokes, NLA:NLA2074}.  The defining feature of monolithic multigrid is the use of coupled relaxation schemes that are crafted to address the block structure of the system, along with a coarse-grid correction procedure that, again, couples the blocks within the system.  Here, we consider geometric multigrid~\cite{MR1807961}, combining coarse-grid correction based on geometric interpolation operators (modified, as discussed below, to account for divergence-free functions) with relaxation that aims to damp oscillatory error components on each grid level.  We write the two-grid error propagation operator as
\begin{equation}\label{TG-Operator-Biot}
	E_{TG}= E_s^{\nu_2}E_{CGC} E_s^{\nu_1},
\end{equation}
where $\nu_1$ and $\nu_2$ are the number of pre- and post-relaxation iterations, respectively. The error-propagation operator for relaxation is $E_s= I -\omega\mathcal{M}^{-1}\mathcal{A}$, where $\omega$ is a damping parameter,
and $E_{CGC}=I-P {\mathcal{A}}_H^{-1}R \mathcal{A}$ for the coarse-grid correction (CGC)
where $P$ is the multigrid interpolation operator and $R$ is the restriction operator. The coarse-grid operator, $\mathcal{A}_H$, is constructed by either rediscretization or as the Galerkin operator, $R\mathcal{A}P$.  As is typical for monolithic multigrid, the interpolation operator is determined block-wise, given as
\begin{equation}\label{eq:block_interp}
	P = \begin{pmatrix}
		P_{\bm u} & 0 & 0 \\
		0 & P_{\bm w} & 0 \\
		0 & 0 & P_p \end{pmatrix},
\end{equation}
where $P_{\bm u}$ is the interpolation operator for displacements, $P_{\bm w}$ is that for the Darcy velocity, and $P_p$ is the interpolation operator for pressure.  We discuss the construction of $P_{\bm u}$ below; for $P_{\bm w}$ and $P_p$, we use the canonical finite-element interpolation operators for RT0 and P0.  We fix $R = P^T$.  While the Galerkin and rediscretization coarse-grid operators coincide when the canonical finite-element operators are used for all fields, they will not do so here, due to the use of the divergence-preserving interpolation for $P_{\bm u}$ discussed below.  Following the geometric multigrid structure, we use the rediscretization operators instead of Galerkin, primarily because this allows easy extension from effective two-level solvers to the multilevel case.

To simplify the notation, we rewrite
\begin{equation}\label{Saddle-point-system}
	\mathcal{A}^{\text{RQ}} \bm{x} =
	\begin{pmatrix}
		A   &B^{\top}\\
		B & -C
	\end{pmatrix}
	\begin{pmatrix} \bm{y}  \\ p\end{pmatrix},
\end{equation}
where
\begin{equation*}
	A=
	\begin{pmatrix}
		A^{\text{RQ}}_{\bm{u}}       &0  \\
		0            &\tau  M_{\bm{\omega}}
	\end{pmatrix},\quad
	B =
	\begin{pmatrix}
		\alpha B_{\bm u}       &\tau B_{\boldsymbol{\omega}}
	\end{pmatrix},\quad
	C = \frac{1}{M} M_p, \text{ and } \bm{y} = \begin{pmatrix} \bm{u} \\ \bm{w} \end{pmatrix}.
\end{equation*}
Next, we detail the non-standard aspects of our multigrid method, namely the\newline divergence-preserving interpolation operator and the coupled relaxation schemes.

\subsection{Divergence-Preserving Interpolation}\label{sec:Pdiv}
As recognized in \cite{Schoberl1998, Schoberl1999a, Schoberl1999b} (see also \cite{farrell2020robust}), a key to achieving solvers for elasticity that are robust in the incompressible (large $\lambda$) limit is the interpolation of divergence-free functions on the coarse mesh to divergence-free functions on the fine mesh.
If $\bm{u}_H$ is a coarse-grid divergence-free function, then, by the divergence theorem,
\begin{equation*}\label{eqn:divfreecoarse}
	\int_{\partial T} \bm{n}^{\top} \bm{u}_H \;d\text{s} = 0, \quad \forall T\in\mathcal{T}_H,
\end{equation*}
where the subscript $H$ denotes the coarse grid whose elements form the set $\mathcal{T}_H$.
Asking that the prolongation of $\bm{u}_H$
to the fine grid also be divergence-free yields,
\begin{equation}\label{eqn:divfreefine}
	\int_{\partial T} \bm{n}^{\top} (P_{\bm u}\bm{u}_H) \;d\text{s} = 0, \quad \forall T\in\mathcal{T}_h,
\end{equation}
where we now impose the condition on the fine-mesh elements in $\mathcal{T}_h$.

The standard finite-element interpolation operator on the displacement space does not satisfy this condition.  To build an operator that does, we consider the interpolation locally from each coarse-grid element, as pictured in \cref{fig:div_free_interp}.  The key step in the construction is to use the finite-element interpolation operator to fix all fine-mesh DoFs on the edges of the coarse-mesh triangle, and use the three edge DoFs on the ``interior'' fine-mesh triangle to enforce \cref{eqn:divfreefine}.  A column-wise construction of the interpolation operator is then given by first computing $\bm{c}_i = \hat{P}_{\bm u}\bm{e}_i$,
where $\bm{e}_i$ is the $i^{\text{th}}$ canonical unit vector on the coarse mesh, and $\hat{P}_{\bm u}$ is the standard finite-element interpolation operator.
Then, the entries in $\bm{c}_i$ that correspond to the interior bubble DoFs depicted in \cref{fig:div_free_interp} are replaced by values that ensure satisfaction of \cref{eqn:divfreefine}.  Consider the triangle,  $t^{1,2,3}$, in \cref{fig:div_free_interp} with
vertices labeled $1,2,3$.  Let $c_b^{v_1,v_2}$ denote the entry in $\bm{c}_i$ associated with the bubble degree of freedom
on the edge between vertices $v_1$ and $v_2$, and let $\bm{c}_v$ denote the
entries in $\bm{c}_i$ associated with the $x$ and $y$ DoFs on vertex $v$.
To make the function represented by $\bm{c}_i$ divergence-free on $t^{1,2,3}$, we set the coefficients of the interior bubble degree of freedom, $c_b^{1,3}$,
to cancel that from the remaining DoFs,
\begin{equation*}
	c_b^{1,3} = - \left(
	c_b^{1,2} + c_b^{2,3} + \frac{1}{|\partial t^{1,2,3}|}
	\sum_{v=1}^{3} \int_{\partial t^{1,2,3}} \bm{n}^T \bm{c}_v \lambda_v \;d\text{s}
	\right),
\end{equation*}
where $\bm{n}$ is the outward normal, and $\lambda_v$ is the linear basis function associated with vertex $v$.
Note that this calculation is simplified by choosing the bubble degrees of freedom to be defined directly as integrals over the associated edges.

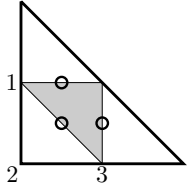
\begin{figure}[h!]
	\begin{center}
		\scalebox{0.9}{
			\begin{tikzpicture}[label distance=-2mm,scale=0.6]
				\def\h{2}
				\node (a) at (\h*0,\h*2) {};
				\node[label=west:{1}] (b) at (\h*0,\h*1) {}; \node (c) at (\h*1,\h*1) {};
				\node[label={[label distance=-3mm]south west:{2}}] (d) at (\h*0,\h*0) {};
				\node[label=south:{3}] (e) at (\h*1,\h*0) {}; \node (f) at (\h*2,\h*0) {};
				\draw[very thick] (a.center) -- (b.center) -- (d.center) -- (e.center) -- (f.center) -- (c.center) -- (a.center);
				\draw (b.center) -- (c.center) -- (e.center) -- (b.center);
				\path[fill=black, opacity=0.2] (b.center) -- (c.center) -- (e.center) -- cycle;
				\draw[line width=1.0pt] ($(b)+\h*(0.5,0)$)   circle (4pt);
				\draw[line width=1.0pt] ($(c)-\h*(0,0.5)$)   circle (4pt);
				\draw[line width=1.0pt] ($(c)-\h*(0.5,0.5)$) circle (4pt);
			\end{tikzpicture}
		}
	\end{center}
	\caption{ A coarse mesh element, $T\in \mathcal{T}_H$, and the four fine-mesh triangles that interpolate from it.
		The circles represent the bubble DoFs that are used to satisfy the divergence-free interpolation condition.
		The gray fine-grid triangle is referred to as the interior triangle,
		while the other fine-grid triangles are the ``corner'' triangles.}
	\label{fig:div_free_interp}
\end{figure}

\subsection{Monolithic Multigrid Relaxation}

It is widely recognized that standard relaxation schemes, such as Jacobi or Gauss-Seidel, are not effective components of a multigrid algorithm for many saddle-point problems \cite{MR1807961}.  Instead, several families of relaxation schemes tailored to this setting have been proposed and studied.  Here, we focus on two classes of such methods, Vanka and Braess-Sarazin relaxation.

\subsubsection{Vanka Relaxation Scheme}\label{subsec:Vanka}
Vanka relaxation, originally proposed in \cite{VANKA1986}, has been adapted for a wide variety of discretizations and saddle-point problems \cite{VolkerLutz2000, LarinReusken2008, JAdler_etal_2015b, MR2840198}.  At its root, Vanka methods are overlapping block relaxation schemes, that can be considered in either additive (block-Jacobi) or multiplicative (block-Gauss-Seidel) form.  While multiplicative variants have long been considered, the additive form has attracted recent interest, due to its natural parallelization \cite{farrell2019c, PFarrell_etal_2019a}.

Given a decomposition of the set of DoFs into $L$ (overlapping) blocks, a standard Schwarz method is most easily defined by defining the restriction operator, $V_{\ell}$, from global vectors to local vectors on block $\ell$.  Then, given a current residual, $\bm{r}^{(j)} = \bm{b} - \mathcal{A}^{\text{RQ}}\bm{x}^{(j)}$, we can solve the projected system
\begin{equation*}
	V_\ell\mathcal{A}^{\text{RQ}}V_\ell^{\top} \hat{\bm{x}}_\ell=V_\ell\vec{r}^{(j)},
\end{equation*}
on each block.
The weighted additive form of the relaxation is then
\begin{equation*}
	\bm{x}^{(j+1)} = \bm{x}^{(j)} +\omega\sum_{\ell}V_{\ell}^{\top}D_{\ell}\hat{\bm{x}}_{\ell},
\end{equation*}
where
$\omega$ is a damping parameter and
$D_{\ell}$ is a diagonal weight matrix that is chosen to compensate for the fact that different (global) DoFs appear in different numbers of patches.
Here, we consider $D_{\ell}$ to be given by the ``natural weights'' of the overlapping block decomposition, where each diagonal entry is equal to the reciprocal of the number of
patches that the corresponding degree of freedom appears in.

The construction of the Vanka blocks is critically important to the success of the resulting multigrid method, with general principles being well-understood for their construction in several contexts \cite{arnold2000multigrid, farrell2019c, MR2840198}.  Following the construction of the reduced quadrature discretization above, our primary concern is in ensuring relaxation suitably handles the locally-supported basis functions for the divergence-free space \cite{Schoberl1998, Schoberl1999a, Schoberl1999b}.  Since those basis functions are supported around the nodes of the mesh, as shown in \cref{fig:DivFreeBasis}, we also use nodal patches for the Vanka blocks, see  \cref{fig:patches}.  For the full poroelasticity system, we use the patches shown at right; those at left will be used within the Braess-Sarazin relaxation scheme discussed next.

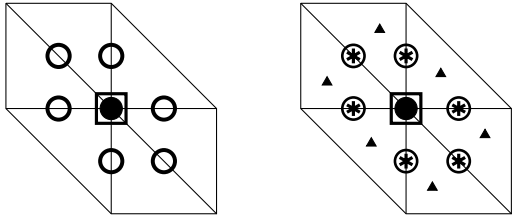
\begin{figure}[h!]
	\begin{center}
		\begin{tikzpicture}[scale=0.7]
			\def\h{2}
			\node (a) at (\h*0,\h*2) {}; \node (b) at (\h*1,\h*2) {};
			\node (c) at (\h*0,\h*1) {}; \node (d) at (\h*1,\h*1) {}; \node (e) at (\h*2,\h*1) {};
			\node (f) at (\h*1,\h*0) {}; \node (g) at (\h*2,\h*0) {};
			\draw (a.center) -- (b.center) -- (e.center) -- (g.center) -- (f.center) -- (c.center) -- (a.center);
			\draw (a.center) -- (d.center) -- (g.center);
			\draw (c.center) -- (d.center) -- (e.center);
			\draw (b.center) -- (d.center) -- (f.center);
			\filldraw (d) circle (6pt);
			\draw[mark=square,very thick, mark size=8pt] plot coordinates{(d)};
			\draw[line width=1.5pt] ($(d)+\h*(0,0.5)$) circle (6pt);
			\draw[line width=1.5pt] ($(d)+\h*(0.5,0)$) circle (6pt);
			\draw[line width=1.5pt] ($(d)-\h*(0,0.5)$) circle (6pt);
			\draw[line width=1.5pt] ($(d)-\h*(0.5,0)$) circle (6pt);
			\draw[line width=1.5pt] ($(d)+\h*(0.5,-0.5)$) circle (6pt);
			\draw[line width=1.5pt] ($(d)+\h*(-0.5,0.5)$) circle (6pt);
		\end{tikzpicture}
		\hspace*{0.05\linewidth}
		\begin{tikzpicture}[scale=0.7]
			\def\h{2}
			\node (a) at (\h*0,\h*2) {}; \node (b) at (\h*1,\h*2) {};
			\node (c) at (\h*0,\h*1) {}; \node (d) at (\h*1,\h*1) {}; \node (e) at (\h*2,\h*1) {};
			\node (f) at (\h*1,\h*0) {}; \node (g) at (\h*2,\h*0) {};
			\draw (a.center) -- (b.center) -- (e.center) -- (g.center) -- (f.center) -- (c.center) -- (a.center);
			\draw (a.center) -- (d.center) -- (g.center);
			\draw (c.center) -- (d.center) -- (e.center);
			\draw (b.center) -- (d.center) -- (f.center);
			\filldraw (d) circle (6pt);
			\draw[mark=square,very thick, mark size=8pt] plot coordinates{(d)};
			\draw[line width=1pt] ($(d)+\h*(0,0.5)$) circle (6pt);
			\draw[line width=1pt] ($(d)+\h*(0.5,0)$) circle (6pt);
			\draw[line width=1pt] ($(d)-\h*(0,0.5)$) circle (6pt);
			\draw[line width=1pt] ($(d)-\h*(0.5,0)$) circle (6pt);
			\draw[line width=1pt] ($(d)+\h*(0.5,-0.5)$) circle (6pt);
			\draw[line width=1pt] ($(d)+\h*(-0.5,0.5)$) circle (6pt);
			
			\draw[mark=asterisk,very thick, mark size=4pt] plot coordinates{($(d)+\h*(0,0.5)$)};
			\draw[mark=asterisk,very thick, mark size=4pt] plot coordinates{($(d)+\h*(0.5,0)$)};
			\draw[mark=asterisk,very thick, mark size=4pt] plot coordinates{($(d)-\h*(0,0.5)$)};
			\draw[mark=asterisk,very thick, mark size=4pt] plot coordinates{($(d)-\h*(0.5,0)$)};
			\draw[mark=asterisk,very thick, mark size=4pt] plot coordinates{($(d)+\h*(0.5,-0.5)$)};
			\draw[mark=asterisk,very thick, mark size=4pt] plot coordinates{($(d)+\h*(-0.5,0.5)$)};
			
			\draw[mark=triangle*,mark size=3pt] plot coordinates{($(d)+\h*(0.33,0.33)$)};
			\draw[mark=triangle*,mark size=3pt] plot coordinates{($(d)+\h*(-0.25,0.75)$)};
			\draw[mark=triangle*,mark size=3pt] plot coordinates{($(d)+\h*(-0.75,0.25)$)};
			\draw[mark=triangle*,mark size=3pt] plot coordinates{($(d)-\h*(0.33,0.33)$)};
			\draw[mark=triangle*,mark size=3pt] plot coordinates{($(d)-\h*(-0.25,0.75)$)};
			\draw[mark=triangle*,mark size=3pt] plot coordinates{($(d)-\h*(-0.75,0.25)$)};
			
		\end{tikzpicture}
	\end{center}
	\caption{Choices of DoFs for blocks within Vanka relaxation on the displacement subsystem (left) and full poroelasticity system (right).  In both figures, filled circles and squares denote the DoFs associated with the linear component of the displacement, while empty circles show the bubble DoFs.  At right, asterisks are used to denote the RT0 DoFs for the Darcy velocity space, and triangles denote the P0 DoFs for the pressure space.}
	\label{fig:patches}
\end{figure}

\subsubsection{Braess-Sarazin Relaxation Schemes}\label{subsec:BSR}
Braess-Sarazin-type algorithms were originally proposed as relaxation schemes for the Stokes' equations \cite{BRAESS1997}, using an approximate block factorization as an approximation to the original system.  Like Vanka relaxation, they have also been extended to many discretizations and systems \cite{VolkerLutz2000, LarinReusken2008, JAdler_etal_2015b, NLA2147, HMFEMStokes}, and are closely related to Uzawa schemes \cite{NLA:NLA2074}.  Using the $2\times 2$ block structure in \cref{Saddle-point-system}, given a residual $\bm{r}^{(j)}$, \textit{exact} Braess-Sarazin relaxation updates the approximation as
\begin{equation}\label{eqn:BSR_standard}
	\begin{pmatrix} \bm{y}^{(j+1)} \\ p^{(j+1)} \end{pmatrix} =
	\begin{pmatrix} \bm{y}^{(j)} \\ p^{(j)} \end{pmatrix} + \omega
	\begin{pmatrix}
		F & B^{\top}\\
		B & -C
	\end{pmatrix}^{-1} \bm{r}^{(j)},
\end{equation}
where $F$ is an approximation of $A$,
often taken to be $\omega_bI$ or $\omega_b{\rm{diag}}(A)$,
with weight $\omega_b$ chosen to separately damp the correction to the variables in $\bm{y}$ from that given by the global parameter, $\omega$.

The matrix inversion in \cref{eqn:BSR_standard} can be carried out in two stages as solving
\begin{align}
	&S\delta p = BF^{-1}\bm{r}^{(j)}_{\bm y}- \bm{r}^{(j)}_{p},& \label{eqn:bsr_p}\\
	&F\delta\bm{y}    = \bm{r}^{(j)}_{\bm y}-B^{\top}\delta p, \nonumber &
\end{align}
where $S= C + BF^{-1}B^{\top}$, and $\bm{r}^{(j)}_{\bm y}$ and $\bm{r}^{(j)}_p$ are the first and second block components of $\bm{r}^{(j)}$ in this decomposition.
In \emph{exact} BSR, there is a significant cost associated with the inversion of the Schur complement, $S$, in \cref{eqn:bsr_p}.  For this reason, inexact BSR methods were proposed,
where the exact solution of the Schur complement equation is replaced by a suitable iterative method applied to \cref{eqn:bsr_p}, typically given by a few steps of a relaxation scheme or of a multigrid cycle for that subsystem.

Here, we make use of the block structure of $A$, to note that
\[
BA^{-1}B^{\top} = \alpha^2 B_{\bm u}\left(A^{\text{RQ}}_{\bm{u}}\right)^{-1}B_{\bm u}^{\top} + \tau B_{\bm w}M_{\bm w}^{-1} B_{\bm w}^{\top},
\]
and that, particularly in the large $\lambda$ limit, $B_{\bm u}\left(A^{\text{RQ}}_{\bm{u}}\right)^{-1}B_{\bm u}^{\top}$ is well-approximated  by a scaled mass matrix on the pressure space. This idea is motivated by the inf-sup condition~\eqref{eqn:stokesA-inf-sup-RQ} and is, essentially, the well-known ``fixed-stress" approximation~\cite{kim2011stability}. Thus, we first approximate
\[
S \approx \frac{1}{M}M_p + \frac{\alpha^2}{\lambda+2\mu/d} M_p + \tau B_{\bm w}D_{\bm w}^{-1}B_{\bm w}^{\top},
\]
where $D_{\bm w}$ is the diagonal of $M_{\bm w}$, and refer to the method with exact inversion of this system in \cref{eqn:bsr_p} as exact BSR.  This is in combination with a single sweep of a Jacobi iteration on $M_{\bm w}$ to approximate the $\bm{w}$ component of $\bm{y}$, and a single iteration of the Vanka relaxation with patches chosen as shown at left of \cref{fig:patches} to approximate the inversion of $A^{\text{RQ}}_{\bm{u}}$ to approximate the $\bm{u}$ component of $\bm{y}$.  For inexact BSR, we replace the exact solve with the approximation to $S$ by a single sweep of weighted Jacobi (with relaxation weight $\omega_J$) on \cref{eqn:bsr_p}.

A downside of these relaxation schemes is their dependence on multiple relaxation parameters in their component parts.  While some general principles exist to help us choose those parameters, often they are fixed by expensive brute-force testing.  Here, we will make use of local Fourier analysis to make these choices.

\section{Local Fourier Analysis}\label{sec:LFA}
LFA is a common and useful
 tool to predict and analyze actual performance of algorithms
for the solution of discretized PDEs \cite{wienands2004practical,
  MR1807961}.  In particular, it has been used to analyze the
construction and optimization of the components of a multigrid algorithm,
such as relaxation schemes and grid-transfer operators \cite{gaspar2017fixed, NLA:NLA2074, NLA2147, HM2018LFALaplace, HMFEMStokes, MR2840198}. In this paper, we
apply the LFA framework developed in \cite{gaspar2017fixed, NLA2147, HMFEMStokes, MR2840198} to  monolithic multigrid methods for the discretized Biot model in \cref{block_form_rq}, in order to optimize the relaxation parameters described above.

\subsection{Two-grid LFA}

Following \cite{MR1807961,wienands2004practical}, we first consider two-dimensional infinite uniform grids, $\mathbf{G}_h=\big\{\bm{x}:=(x_1,x_2)=(k_{1},k_{2})h,\quad (k_1,k_2)\in \mathbb{Z}^2\big\}$.
Let $L_h$ be a scalar Toeplitz operator defined as $\displaystyle L_{h}w_{h}(\bm{x})=\sum_{\bm{\kappa}\in\bm{S}}s_{\bm{\kappa}}w_{h}(\bm{x}+\bm{\kappa}h)$, $\bm{\kappa}=(\kappa_{1},\kappa_{2})\in \bm{S}$,
with constant coefficients $s_{\bm{\kappa}}\in \mathbb{R} \,(\textrm{or} \,\,\mathbb{C})$, and where $w_{h}(\bm{x})$ is
a function in $l^2(\mathbf{G}_{h})$. Here, $\bm{S}\subset \mathbb{Z}^2$ is a finite index set over which the stencil is nonzero.  Because $L_h$ is formally diagonalized by the Fourier modes $\varphi(\bm{\theta},\bm{x})= e^{\iota\bm{\theta}\cdot\bm{x}/{h}}=e^{\iota \theta_1x_1/h}e^{\iota \theta_2x_2/h}$, where $\bm{\theta}=(\theta_1,\theta_2)$, we use $\varphi(\bm{\theta},\bm{x})$ as a Fourier basis with $\bm{\theta}\in \big[-\frac{\pi}{2},\frac{3\pi}{2}\big)^{2}$ (or any pair of intervals with length $2\pi$). High and low frequencies for standard coarsening (as considered here) are given by
\begin{equation*}
  \bm{\theta}\in T^{{\rm low}} =\left[-\frac{\pi}{2},\frac{\pi}{2}\right)^{2}, \, \bm{\theta}\in T^{{\rm high}} =\displaystyle \left[-\frac{\pi}{2},\frac{3\pi}{2}\right)^{2} \bigg\backslash \left[-\frac{\pi}{2},\frac{\pi}{2}\right)^{2}.
\end{equation*}
\begin{definition}\label{formulation-symbol-classical}
If for all grid functions $\varphi(\bm{\theta},\bm{x})$,
$L_{h}\varphi(\bm{\theta},\bm{x})= \widetilde{L}_{h} (\bm{\theta})\varphi(\bm{\theta},\bm{x}),$
 we call $\widetilde{L}_{h}(\bm{\theta})=\displaystyle\sum_{\bm{\kappa}\in\bm{S}}s_{\bm{\kappa}}e^{\iota \bm{\theta}\cdot\bm{\kappa}}$ the symbol of $L_{h}$.
\end{definition}
 
For simple scalar operators (such as second-order finite-difference or finite-element discretizations of constant-coefficient diffusion equations), the performance of a standard relaxation method, such as the weighted Jacobi or Gauss-Seidel iterations, is easily analyzed by considering the symbol of the relaxation scheme \cite{wienands2004practical, MR1807961}.  From the heuristic argument that coarse-grid correction effectively reduces error in $T^{{\rm low}}$, the LFA smoothing factor for a relaxation scheme with error-propagation operator given by $I-\omega M_h^{-1}L_h$ is introduced as
$\mu = \sup_{\bm{\theta}\in T^{{\rm high}}} \left|1-\omega\widetilde{M}_h(\bm{\theta})^{-1}\widetilde{L}_h(\bm{\theta})\right|$,
where $\omega$ is a damping parameter.

While the LFA smoothing factor provides excellent predictions of true multigrid performance for simple discretizations of simple operators, it is known to provide poor predictions when used on complicated or higher-order operators \cite{HM2018LFALaplace}.  In such settings, it is more reliable to use the two-grid LFA convergence factor, which takes into account the coarse-grid correction process.  To do this, we define the harmonic modes by taking $\boldsymbol{\theta}^{\boldsymbol{\alpha}} = (\theta_1^{\alpha_1},\theta_2^{\alpha_2})=\boldsymbol{\theta}^{00}+\pi\cdot\boldsymbol{\alpha}$, $\boldsymbol{\alpha}=(\alpha_1,\alpha_2)\in\big\{(0,0),(1,0),(0,1),(1,1)\big\}$ and $\boldsymbol{\theta}^{00}\in T^{{\rm low}}$.
That is, for each low-frequency mode $\bm \theta\in T^{\rm low}$, we define a four-dimensional harmonic space,
$  \mathcal{F}(\bm \theta)={\rm span}\Big\{ \varphi(\boldsymbol{\theta^{\alpha}},\cdot): \boldsymbol{\alpha}\in\big\{(0,0),(1,0),(0,1),(1,1)\big\}\Big\}$,
which is invariant for standard full-coarsening two-grid algorithms.

To compute the LFA two-grid convergence factor, we must obtain an LFA representation of all components of the multigrid cycle.  This requires finding symbols for not just the fine-grid operator and relaxation scheme, but also for the interpolation and restriction operators, and for the coarse-grid operator.  The \textit{symbol} of the two-grid algorithm is a $4\times 4$ matrix that describes the action of the two-grid algorithm, and comes from noting that structured constant-coefficient interpolation and restriction operators map naturally between the four fine-grid harmonic modes in $\mathcal{F}(\bm \theta)$ and the coarse-grid mode $2{\bm \theta}$.  Writing $\widetilde{L}_{2h}$ for the symbol of the coarse-grid operator and $\widetilde{P}_h$ and $\widetilde{R}_h$ for the symbols of the interpolation and restriction operators, the Fourier representation of the two-grid error-propagation operator is defined as
\begin{equation*}
 \widetilde{\boldsymbol{E}}_{TG}(\boldsymbol{\theta})= \widetilde{\bm{E}}_s^{\nu_2}(\boldsymbol{\theta})\big(I-\widetilde{\boldsymbol {P}}_h(\boldsymbol{\theta})(\widetilde{{L}}_{2h}(2\boldsymbol{ \theta}))^{-1}\widetilde{\boldsymbol{ R}}_h(\boldsymbol{\theta})\widetilde{\boldsymbol{ L}}_{h}(\boldsymbol{\theta})\big)\widetilde{\boldsymbol{E}}_s^{\nu_1}(\boldsymbol{\theta}),
\end{equation*}
where
\begin{eqnarray*}
\widetilde{\bm{L}}_h(\bm{\theta})&=&\text{diag}\left\{\widetilde{{L}}_h(\bm{\theta}^{00}), \widetilde{{L}}_h(\bm{\theta}^{10}),\widetilde{{L}}_h(\bm{\theta}^{01}),
\widetilde{{L}}_h(\bm{\theta}^{11})\right\},\\
\widetilde{\bm{E}}_s(\bm{\theta})&=&\text{diag}\left\{\widetilde{E}_s(\bm{\theta}^{00}),
\widetilde{E}_s(\bm{\theta}^{10}),\widetilde{E}_s(\bm{\theta}^{01}),
\widetilde{E}_s(\bm{\theta}^{11})\right\},\\
\widetilde{\bm{R}}_h(\bm{\theta})&=&\left(\widetilde{R}_h(\bm{\theta}^{00}),\widetilde{R}_h(\bm{\theta}^{10}),
\widetilde{R}_h(\bm{\theta}^{01}),\widetilde{R}_h(\bm{\theta}^{11}) \right),\\
\widetilde{\bm{P}}_h(\bm{\theta})&=&\left(\widetilde{P}_h(\bm{\theta}^{00});\widetilde{P}_h(\bm{\theta}^{10});
\widetilde{P}_h(\bm{\theta}^{01});\widetilde{P}_h(\bm{\theta}^{11}) \right).
\end{eqnarray*}
Here, ${\rm diag}\{T_1,T_2,T_3,T_4\}$ denotes the block diagonal matrix with diagonal blocks, $T_1, T_2, T_3$, and $T_4$ \cite{MR1807961,wienands2004practical}.
With this, we define the two-grid LFA convergence factor.
\begin{definition}
The two-grid LFA convergence factor, $\rho_{{\rm LFA}}$, is defined as
\begin{equation}\label{real-TGM}
  \rho_{LFA} = {\rm sup}\{\rho(\widetilde{\boldsymbol{E}}_{TG}(\boldsymbol{\theta}): \boldsymbol{\theta}\in T^{{\rm low}}\},
\end{equation}
where $\rho(\widetilde{\boldsymbol{E}}_{TG}(\boldsymbol{\theta}))$ denotes the spectral radius of matrix $\widetilde{\boldsymbol{E}}_{TG}(\boldsymbol{\theta})$.
\end{definition}

As described above, it is natural to introduce algorithmic parameters when designing multigrid methods for complicated problems.  It is for this purpose that we introduce LFA here.  While it is often possible to optimize the LFA smoothing factor for simple problems through analytical means (see, for example, \cite{NLA2147}), optimizing the two-grid LFA convergence factor for more complicated problems and algorithms is a challenging task \cite{LFAoptAlg}.  Here, we will develop LFA representations of the monolithic multigrid algorithms above, and optimize the two-grid convergence factor in \cref{real-TGM} using brute-force sampling.  In particular, while the true two-grid LFA convergence factor is most naturally defined as a supremum over a continuous range of values of $\bm{\theta}$, we will use a discrete sampling at a finite number of evenly-spaced frequencies in the domain $(-\frac{\pi}{2},\frac{\pi}{2}]^2$, but without any change of notation.

\subsection{LFA Representation of Discretized System}\label{subsec:LFA-Biot}

To extend Fourier analysis to the full discretized system in \eqref{block_form_rq}, we must account for the fact that the system is not readily extended to a Toeplitz operator on an infinite grid, unlike in the scalar case.  This occurs in two ways.  First, as is clear, the discretization of a coupled system of PDEs leads, at best, to a block operator with Toeplitz blocks.  Secondly, even within a single block, such as $A_{\bm u}^{\text{RQ}}$, there are different ``types'' of DoFs, leading to nested block-Toeplitz structure.

The key concept in enabling LFA is in expressing the block-Toeplitz structure of the multigrid hierarchy and relaxation operator relative to the infinite grid, $\bm{G}_h$.  With triangular cells and face- and cell-based DoFs, this is slightly non-intuitive.  \cref{fig:DoF_structure2} shows the DoFs in a typical pair of elements on the mesh, constructed by ``cutting'' a quadrilateral cell into two triangles.  With this arrangement of DoFs, we have natural periodic structure for the P1 components of the displacement (2 DoFs, 1 for each component of the 2D displacement vector, $\bm{u}$), but also for the 6 face-based DoFs, coming in two pairs of 3 DoFs, corresponding to the normal displacement bubble component along each face and the face-based Raviart-Thomas DoFs for the Darcy velocity.  Note that we do not ``collapse'' the Fourier representation of the face-based DoFs to a single component within the symbol; this is not possible, since the matrix connections between face-based DoFs along (for example) horizontal edges will be different than those along diagonal edges.  Instead, we will maintain an entry in the Fourier symbol for each ``type'' of face-based DoF. Similarly, the connections between the P0 DoFs in the lower-left triangles and the other variables in the cell may be different than those with the P0 DoFs in the upper-right triangles.  Thus, we introduce Fourier representations of both of these DoFs.  In total, this yields a $10\times 10$ block Fourier symbol for the operator, $\mathcal{A}^{\text{RQ}}$.  With this structure, it is a straightforward (but tedious) task to compute the Fourier symbol of $\mathcal{A}^{\text{RQ}}$.  We outline the main ideas here, but leave the technical details as Supplementary Material for the interested reader.

\begin{figure}[h!]
\begin{center}
\begin{tikzpicture}[scale=0.65]
    \draw [-] [thick] (0,0) --(4,0);
    \draw [-] [thick] (0,0) --(0,4);
    \draw [-] [thick] (0,4) --(4,4);
    \draw [-] [thick] (4,0) --(4,4);
    \draw [-] [thick] (0,4) --(4,0);

    \node (d1) at (0,0) {};
    \node (d2) at (0,4) {};
    \node (d3) at (4,0) {};
    \node (d4) at (4,4) {};
    \filldraw (d1) circle (6pt);
    \filldraw (d2) circle (6pt);
    \filldraw (d3) circle (6pt);
    \filldraw (d4) circle (6pt);

    \draw[line width=1pt] (2,0) circle (6pt);
    \draw[line width=1pt] (2,4) circle (6pt);
    \draw[line width=1pt] (0,2) circle (6pt);
    \draw[line width=1pt] (4,2) circle (6pt);
    \draw[line width=1pt] (2,2) circle (6pt);

    \draw[mark=asterisk,very thick, mark size=4pt] plot coordinates{(2,0)};
    \draw[mark=asterisk,very thick, mark size=4pt] plot coordinates{(2,4)};
    \draw[mark=asterisk,very thick, mark size=4pt] plot coordinates{(0,2)};
    \draw[mark=asterisk,very thick, mark size=4pt] plot coordinates{(4,2)};
    \draw[mark=asterisk,very thick, mark size=4pt] plot coordinates{(2,2)};

    \draw[mark=triangle*,mark size=6pt] plot coordinates{(1,1)};
    \draw[mark=triangle*,mark size=6pt] plot coordinates{(3,2.5)};

    \draw[mark=square,very thick, mark size=8pt] plot coordinates{(0,0)};
    \draw[mark=square,very thick, mark size=8pt] plot coordinates{(4,0)};
    \draw[mark=square,very thick, mark size=8pt] plot coordinates{(0,4)};
    \draw[mark=square,very thick, mark size=8pt] plot coordinates{(4,4)};

\end{tikzpicture}
\end{center}
\caption{Cut quadrilateral mesh cell showing DoFs.  Filled circles and squares denote the DoFs associated with the linear component of the displacement, while empty circles show the bubble DoFs.  Asterisks are used to denote the RT0 DoFs for the Darcy velocity space, and triangles denote the P0 DoFs for the pressure space.}
\label{fig:DoF_structure2}
\end{figure}
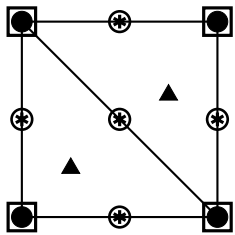

First, we ``expand'' its block structure from the canonical $3\times 3$ form to that of a $10\times 10$ block-structured linear system, with 1 block for each DoF identified above and in  \cref{fig:DoF_structure2}.  In this ordering, each diagonal block is a Toeplitz matrix, whose symbol can be calculated according to \cref{formulation-symbol-classical}.  Off-diagonal blocks in this structure are also Toeplitz matrices, although we also account for the offsets between DoF locations in the mesh in the Fourier symbols, in a similar manner to what was done in \cite{NLA2147, HMFEMStokes}.  Details of these calculations are presented in \Cref{secSM:LFA-RQD}.

Similarly, Fourier representations of the grid-transfer operators can also be computed in block form.  Taking the block-diagonal interpolation operator from \cref{eq:block_interp}, we separately compute Fourier representations of each interpolation operator, accounting for block structure of the DoFs and the details of the interpolation schemes.  Since we have a 10-dimensional space associated with each Fourier frequency, and interpolation and restriction map between four harmonic frequencies on the fine mesh and a single frequency on the coarse mesh, this results in a $40\times 10$ symbol for interpolation and a $10\times 40$ symbol for restriction, which can be broken into $10\times 10$ blocks giving the part of the symbol associated with each individual frequency in the harmonic set.  These $10\times 10$ blocks can be broken down further, based on the block-diagonal form in \cref{eq:block_interp}, to a $5\times 5$ block associated with displacements, a $3\times 3$ block for Darcy velocities, and a $2\times 2$ block for pressures.  It is somewhat more natural to compute Fourier representations of the restriction operators, and use (scaled) transposes of these symbols for interpolation, which is the approach followed in \Cref{secSM:LFA-transfer-operator}.

Finally, Fourier representations of the relaxation schemes can be computed.  For Vanka relaxation, this follows the approach presented in \cite{PFarrell_etal_2019a}, where the Fourier representation of a residual at given frequency is restricted, via $V_\ell$, to a Vanka patch, and the action of the local solve is computed exactly on this basis, with accounting for the overlap between patches.  Details are given in \Cref{secSM:LFA-Vanka}.   For Braess-Sarazin relaxation, the symbols of $F$, $S$, and the approximation to $S$ are readily computed in the same manner as the symbols above, and the incorporation of a relaxation scheme in place of an exact inversion of $S$ is done similarly. See \Cref{secSM:LFA-BSR} for the details.

\subsection{Validation and Optimization}

While we are primarily interested in the use of monolithic multigrid as a preconditioner for GMRES, we begin by studying its use as a stationary iteration, for the purposes of optimizing parameters in the methods.  We use LFA to predict convergence factors associated with given choices of parameters, and compare to measured performance of a stationary iteration, approximating the asymptotic convergence factor of the iteration as $\rho_{N} = \frac{\|\bm{r}^{(j)}\|}{\|\bm{r}^{(j-1)}\|}$,
where $\bm{r}^{(j)}$ is the residual at the $j$-th iteration.  To ensure a good approximation of the asymptotic convergence factor, iterations are run until the change in the measured convergence factor
between iterations is less than $10^{-3}$.
While LFA can be made exact in the case of periodic boundary conditions, the numerical tests
were performed using Dirichlet boundary conditions as is more common.
We consider $\Omega = [0,1]^2$ covered with a uniform
triangular grid with mesh spacing $h=1/64$.
As a test problem, we consider a zero right-hand side,
with a random initial guess for a single time step with $\tau = 1$.
To demonstrate the impact of the physical parameters, the permeability, $\bm{K}$,
and the Poisson ratio $\nu$ are varied.
In all test cases, we consider a diagonal permeability tensor $\bm{K} = k\bm{I}$.
Additionally, $\alpha=1$, $\mu_f = 1$, $M = 10^6$,
and $E= 3\times10^4$. LFA is performed using 32 evenly-spaced sample points in each coordinate direction, offset so that no sample is taken at the origin in Fourier space.  Note that the two-grid LFA convergence factor, \cref{real-TGM}, is a function of the damping parameter, $\omega$. In order to obtain an efficient algorithm, we use brute-force sampling to optimize the LFA-predicted two-grid convergence factors over choices of $\omega$, with steps of size $0.02$.

\begin{table}[htp!]
 	\footnotesize
	\begin{center}
		\caption{Optimized relaxation parameter $(\omega_{\rm opt}$), observed convergence factor ($\rho_N$) with Dirichlet boundary conditions, and optimal two-grid LFA predictions ($\rho_{LFA}$) for additive Vanka relaxation ($\nu_1=\nu_2=2$) on the full system with the 20-DoF vertex-based patch
                 (\cref{fig:patches}, right), varying $k$ and $\nu$.}
        \label{tab:vanka22}
		\begin{tabular}{|l |l | l l l l l l|}
			\hline 
            \multicolumn{2}{|c|}{\backslashbox{$\nu$}{$k$}} & $1$ & $10^{-2}$ & $10^{-4}$ & $10^{-6}$ & $10^{-8}$ & $10^{-10}$\\
			\hline\hline
				$\nu=0$& $\omega_{\rm opt}$         &0.92   &0.92      &0.92    &0.92      &0.88     &0.76 \\\hline
            \multicolumn{2}{|c|}{$\rho_{LFA}$}      &0.705   &0.705    &0.705    &0.702   &0.490    &0.552   \\
            \multicolumn{2}{|c|}{${\rho}_N$}        &0.722   &0.722    &0.722    &0.722   &0.475    &0.547   \\
            \hline\hline
			$\nu=0.2$& $\omega_{\rm opt}$           &0.90      &0.90    &0.90       &0.90     &0.86      &0.76       \\\hline
            \multicolumn{2}{|c|}{$\rho_{LFA}$}      &0.624     &0.624   &0.624      &0.622     &0.474    &0.557     \\
            \multicolumn{2}{|c|}{${\rho}_N$}        &0.610	   &0.610	&0.610	    &0.611	  &0.468	&0.552   \\
            \hline\hline
			$\nu =0.4$& $\omega_{\rm opt}$         &0.80    &0.80     &0.80     &0.80     &0.78     &0.76      \\\hline
            \multicolumn{2}{|c|}{$\rho_{LFA}$}     &0.410    &0.410    &0.410    &0.410    &0.436    &0.562     \\
            \multicolumn{2}{|c|}{${\rho}_N$}      &0.403	&0.403	  &0.403	&0.404	  &0.432	&0.557   \\
            \hline\hline
			$\nu=0.45$& $\omega_{\rm opt}$       &0.76     &0.76   &0.76     &0.76     &0.76     &0.76    \\\hline
            \multicolumn{2}{|c|}{$\rho_{LFA}$ }  &0.492    &0.492  &0.492    &0.492    &0.498    &0.564          \\
            \multicolumn{2}{|c|}{${\rho}_N$}     &0.489	   &0.489  &0.489    &0.489   &0.495     &0.560 \\
            \hline\hline
			$\nu=0.49$& $\omega_{\rm opt}$       &0.74    &0.74    &0.74   &0.74     &0.74      &0.76    \\\hline
            \multicolumn{2}{|c|}{$\rho_{LFA}$}   &0.572   &0.572   &0.572  &0.572    &0.572     &0.573        \\
            \multicolumn{2}{|c|}{${\rho}_N$}     &0.569   &0.569   &0.569  &0.570    &0.570     &0.569  \\
            \hline\hline
			$\nu =0.499$& $\omega_{\rm opt}$    &0.72     &0.72       &0.72    &0.72      &0.72     &0.74     \\\hline
            \multicolumn{2}{|c|}{$\rho_{LFA}$}  &0.600    & 0.600     &0.600   &0.600     &0.600    &0.599             \\
            \multicolumn{2}{|c|}{${\rho}_N$}    &0.596    &0.596      &0.596   &0.596     &0.596    &0.596 \\
			\hline
		\end{tabular}
	\end{center}
\end{table}

\begin{table}[htp!]
	\footnotesize
	\begin{center}
		\caption{Optimized relaxation parameters ($\omega_{J,\rm opt}$, $\omega_{\rm opt}$), observed convergence factor ($\rho_N$) with Dirichlet boundary conditions, and optimal two-grid LFA
         predictions ($\rho_{LFA}$) for inexact Braess-Sarazin relaxation (using one sweep of damped Jacobi
         for the approximate solve of the Schur complement and additive Vanka for the displacement block), varying k and $\nu$.}
                \label{tab:bsr_vanka11}

		\begin{tabular}{|l |l | l l l l l l|}
			\hline 
            \multicolumn{2}{|c|}{\backslashbox{$\nu$}{$k$}} & $1$ & $10^{-2}$ & $10^{-4}$ & $10^{-6}$ & $10^{-8}$ & $10^{-10}$\\
			\hline\hline
				$\nu=0$& $\omega_{J,\rm opt}$         &1.10     &1.15      &1.06     &1.26     &0.98       &0.96   \\
                 & $\omega_{ \rm opt}$                &0.72     &0.68       &0.76    &0.62      & 0.88      &0.98        \\\hline
            \multicolumn{2}{|c|}{$\rho_{LFA}$}        &0.648    &0.649      &0.656   & 0.645     &0.556      &0.417      \\
            \multicolumn{2}{|c|}{${\rho}_N$}          &0.647	&0.646      &0.650	&0.636	&0.535	&0.552\\
            \hline\hline
			$\nu=0.2$& $\omega_{J,\rm opt}$          &1.30    &1.23     &1.00  &0.97      &1.14       &0.95     \\
              & $\omega_{\rm opt}$                   &0.60    & 0.62     &0.71     &0.82      &0.73       &1.07        \\\hline
            \multicolumn{2}{|c|}{$\rho_{LFA}$}       &0.660   &0.652     &0.683   &0.647      &0.568      &0.435      \\
            \multicolumn{2}{|c|}{${\rho}_N$}         &0.652   &0.620     &0.682   &0.616      &0.578      &0.418\\
            \hline\hline
			$\nu =0.4$& $\omega_{J,\rm opt}$       &1.27       &1.16      &1.17   &0.94     &1.18     &0.79   \\
                      & $\omega_{\rm opt}$         &0.69       &0.74      &0.74   & 0.74    &0.77     & 1.14       \\\hline
            \multicolumn{2}{|c|}{$\rho_{LFA}$}     &0.684      &0.663     &0.670  &0.690    &0.659    &0.509     \\
            \multicolumn{2}{|c|}{${\rho}_N$}       &0.680      &0.655     &0.659  &0.654    &0.658    &0.507  \\
            \hline\hline
			$\nu=0.45$& $\omega_{J,\rm opt}$       &1.16    &1.07       &1.10    &0.91      &1.26    &0.76      \\
                     & $\omega_{\rm opt}$          &0.72    &0.72       &0.72    &0.72      & 0.74   &1.17       \\\hline
            \multicolumn{2}{|c|}{$\rho_{LFA}$ }    &0.732   &0.732      &0.732   &0.732      & 0.732  &0.570      \\
            \multicolumn{2}{|c|}{${\rho}_N$}       &0.723	&0.723     &0.723    &0.722     &0.731    &0.567   \\
            \hline\hline
			$\nu=0.49$& $\omega_{J,\rm opt}$        &1.21      &1.27      &1.07    &1.18    &1.29      &1.06     \\
              & $\omega_{\rm opt}$                  &0.69      &0.69      &0.69    &0.69    &0.69      &1.00       \\\hline
            \multicolumn{2}{|c|}{$\rho_{LFA}$}      &0.772     &0.772     &0.772   &0.772   &0.772     &0.688     \\
            \multicolumn{2}{|c|}{${\rho}_N$}        &0.757     &0.757     &0.757   &0.757   &0.741     &0.681    \\
            \hline\hline
			$\nu =0.499$& $\omega_{J,\rm opt}$     &0.85     &1.29     &1.00    &0.86      &0.85     &1.42       \\
                    & $\omega_{\rm opt}$            &0.68     &0.68     & 0.68  &0.68      &0.68      &0.73     \\\hline
            \multicolumn{2}{|c|}{$\rho_{LFA}$}     & 0.786     &0.786   &0.786  &0.786     &0.786     &0.772       \\
            \multicolumn{2}{|c|}{${\rho}_N$}       & 0.783    &0.783    &0.783  &0.783     &0.783     &0.751 \\
			\hline
		\end{tabular}
	\end{center}
\end{table}

In~\cref{tab:vanka22,tab:bsr_vanka11}, we present LFA-optimized parameters and both LFA-predicted and numerically measured two-grid convergence factors for monolithic multigrid using Vanka (with $\nu_1=\nu_2=2$) and inexact Braess-Sarazin relaxation schemes (with $\nu_1=\nu_2=1$), respectively. To validate the parameters for inexact BSR, we first perform LFA for the exact BSR scheme discussed above (not shown here).  For values of $\nu$ larger than 0.4, we find identical performance between exact and inexact BSR, except for the case of $k=10^{-10}$, where inexact BSR slightly outperforms exact BSR for $\nu = 0.45$.  Exact BSR performance notably improves as $\nu$ decreases, achieving  convergence factors around $0.48$ for $\nu = 0$ and larger values of $k$.  While this is a slight improvement in convergence over the inexact BSR case, it relies on the prohibitively expensive exact inversion of the approximate Schur complement.  Note that we also optimize for the jacobi weight, $\omega_J$, for approximately solving the Schur complement.

In general, we see good agreement between the LFA predictions and the measured factors, and that the two-grid schemes are robust to both the incompressible limit, $\nu \rightarrow 0.5$, and extremely small values of $k$.  We note some irregularity in both the convergence factors themselves and the match between prediction and measurement in the small $k$ limit, which appears to be due to ill-conditioning of the Fourier symbols when $k$ is so small.  This also leads to some irregularity in the optimal relaxation parameters also in this limit.

In these tests, we focus on the optimization of only the outer relaxation parameter, $\omega$, using LFA.  While it is possible to introduce more relaxation parameters (e.g., in the inner Vanka relaxation for inexact BSR, or the weighting matrix, $D_\ell$), preliminary experiments showed that these did not greatly improve convergence.  It is also important to note both that the optimal relaxation parameter varies with $\nu$ and that good choices for one value of $\nu$ do not lead to good performance across all values considered here.  With Vanka relaxation, for fixed $\omega = 0.9$ (close to the optimal value for $\nu = 0$), we see divergence for all tested values of $\nu > 0.2$.  For fixed $\omega = 2/3$ (close to the optimal value for $\nu \rightarrow 1/2$), we see strong degradation in convergence as $\nu$ gets small, with divergence for all tested values of $\nu < 0.49$.  We also note that, because these relaxation weights are used in multiplicative combination with coarse-grid correction, the performance of multigrid-preconditioned FGMRES, as is considered in \Cref{sec:Numrical}, is also sensitive to these choices.

\section{Numerical Results}\label{sec:Numrical}

We now consider performance of the reduced quadrature discretization and the monolithic multigrid preconditioners, extending the two-level results shown above to the multilevel case.  To allow fair comparison between the relaxation schemes, we have implemented both Vanka and inexact BSR in a single codebase, namely the
HAZmath package~\cite{hazmath}: a simple finite element, graph, and solver library. All timed numerical results are done using a
workstation with an 8-core 3-GHz Intel Xeon Sandy Bridge CPU and 256 GB of RAM.  This also allows direct comparison to timings for the block preconditioners from \cite{adler2019robust}.

\subsection{Steady-State Model}\label{sec:comparison}
Here, we use a single four-level V-cycle of the monolithic multigrid method as a preconditioner for FGMRES using a relative residual stopping tolerance of $10^{-6}$ and compare the performance with the block upper-triangular preconditioner previously used in \cite{adler2019robust}, with form
\begin{equation}\label{upper-preconditioner}
	\mathcal{B}_U = \left(\begin{array}{ccc}
		A_{\bm u}^{\text{RQ}} & \alpha B_{\bm u}^{\top} & 0 \\
		0 & \left(\frac{\alpha^2}{\zeta^2} + \frac{1}{M}\right) M_p & -\tau B_{\bm w}  \\
		0 & 0 & \tau M_{\bm w} + \tau^2 \left(\frac{\alpha^2}{\zeta^2}+\frac{1}{M}\right)^{-1} A_{\bm w}
	\end{array}\right).
\end{equation}
Notice that \cref{upper-preconditioner} is applied to a permuted form of the discretization, as was considered in \cite{adler2019robust}.
Similar to \cite{adler2019robust},
each diagonal block in the preconditioner is solved to a relative residual tolerance of $10^{-3}$
using preconditioned FGMRES preconditioned with algebraic multigrid for the pressure and Darcy blocks and
FGMRES preconditioned with geometric multigrid using the Vanka relaxation presented in \Cref{subsec:Vanka} for the displacement block.

In this example, the right-hand side functions $\bm{g}$ and $f$ are chosen so that the exact solution is given by
\begin{eqnarray*}
	\bm{u}(x,y,t) &=& \operatorname{curl} \varphi =
	\begin{pmatrix}\partial_y \varphi \\ -\partial_x \varphi \end{pmatrix},
	\quad \varphi (x,y) = [xy(1-x)(1-y)]^2, \\
	p(x,y,t) &=& 1, \qquad \vec{w}(x,y,t) = \vec{0}.
\end{eqnarray*}
The material parameters are  the same as those used in the LFA validation above.
Finally, starting with a zero initial guess, we set $\tau = 1$ and $t_{\max{}}=1$, so that we only perform one time step,
and fix the mesh spacing to be $h=1/64$ (the four-level V-cycle has a direct solve on the coarse mesh with spacing $h=1/8$).
\cref{tab:time_all_redisc2} presents results for monolithic multigrid with both Vanka and inexact Braess-Sarazin relaxation, and for the block preconditioner.

\begin{table}[h!]
	\footnotesize
	\begin{center}
		\caption{CPU time in seconds (and iterations to convergence) for FGMRES preconditioned by monolithic multigrid with additive Vanka relaxation on the full system with the 20-DoF vertex-based patch
			(\cref{fig:patches}, right) and inexact BSR, and preconditioned by a block-upper triangular system from \cite{adler2019robust} on steady-state problem.}
		\label{tab:time_all_redisc2}
		\begin{tabular}{|l||p{0.35in} | p{0.51in} p{0.51in} p{0.51in} p{0.51in} p{0.51in} p{0.51in}|}
			\hline 
			\hspace{-4pt}Scheme\hspace{-4pt}&\backslashbox{$\nu$}{$k$} & $1$ & $10^{-2}$ & $10^{-4}$ & $10^{-6}$ & $10^{-8}$ & $10^{-10}$\\
			\hline
			\hline
			Vanka&$0.0$&2.417 (18) &1.966 (18) &2.037 (18) &2.049 (18) &1.029 (10) &1.002 (10) \\
			BSR&$0.0$&0.457 (9) &0.458 (9) &0.451 (9) &0.613 (12) &0.502 (10) &0.459 (9) \\
			Block& $0.0$  &0.733 (16) &0.738 (16) &0.845 (16) &0.611 (13) &0.499 (12) &0.415 (8) \\
			\hline
			Vanka&$0.2$&1.863 (15) &1.605 (15) &1.544 (15) &1.550 (15) &1.025 (10) &1.103 (10) \\
			BSR&$0.2$&0.615 (12) &0.561 (11) &0.667 (10) &0.556 (11) &0.454 (9) &0.507 (10) \\
			Block&	$0.2$  &0.695 (15) &0.698 (15) &0.718 (15) &0.637 (12) &0.484 (11) &0.448 (8) \\
			\hline
			Vanka&$0.4$&1.095 (9) &1.027 (9) &1.000 (9) &0.923 (9) &1.006 (9) &1.089 (10) \\
			BSR&$0.4$&0.554 (11) &0.608 (12) &0.758 (15) &0.664 (13) &1.124 (22) &0.658 (13) \\
			Block&	$0.4$  &0.784 (15) &0.785 (15) &0.834 (15) &0.850 (13) &0.596 (11) &0.516 (9) \\
			\hline
			Vanka&$0.45$&1.091 (9) &0.921 (9) &1.003 (9) &0.922 (9) &1.018 (10) &0.944 (10) \\
			BSR&$0.45$&0.658 (13) &0.844 (13) &0.921 (14) &0.763 (15) &28.5 (452) &0.819 (16) \\
			Block&$0.45$ &1.047 (15) &1.047 (15) &1.521 (16) &1.443 (15) &0.876 (12) &0.742 (11) \\
			\hline
			Vanka&$0.49$&1.347 (11) &1.142 (11) &1.121 (11) &1.128 (11) &1.125 (11) &0.957 (11) \\
			BSR&$0.49$&0.917 (18) &0.911 (18) &0.971 (19) &0.966 (19) &17.4 (299) &1.126 (22) \\
			Block&$0.49$ &1.081 (15) &1.080 (15) &1.073 (15) &1.833 (16) &1.184 (12) &1.094 (11) \\
			\hline
			Vanka&$0.499$&1.565 (13) &1.354 (13) &1.318 (13) &1.318 (13) &1.437 (13) &1.187 (14) \\
			BSR&$0.499$&1.221 (24) &1.163 (23) &1.220 (24) &1.656 (26) &1.341 (24) &2.808 (55) \\
			Block&$0.499$&2.142 (15) &2.140 (15) &2.146 (15) &2.144 (15) &2.194 (16) &2.099 (15) \\
			\hline
		\end{tabular}
		
	\end{center}
\end{table}

There are several takeaways from these results.  First, the monolithic multigrid with Vanka relaxation method is robust with respect to the physical parameters, though we do see a slight degradation for small $\nu$ (i.e., the compressible case). This is not surprising, as the methods developed here were developed specifically for the limit as $\nu$ approaches $1/2$.  Secondly, the multilevel monolithic multigrid with inexact BSR relaxation struggles when the permeability constant, $k$ is small, in contrast to the robust two-level results in  \Cref{sec:LFA}.  It may be that W-cycles, or other approaches, are needed to achieve robustness in this case, but we do not investigate this here.  Note, however, that for larger $k$, the total computational time when using inexact BSR is slightly faster than that for Vanka relaxation.
Comparing the monolithic multigrid performance with that of the block preconditioner, we see that the block preconditioner performance is similar to that of the monolithic multigrid with inexact Braess-Sarazin relaxation for small $\nu$.
However, there is a clear degradation in performance of the block preconditioner in the incompressible limit, where monolithic multigrid is more robust.  Since the degradation in CPU time is much worse than that in iteration count, we infer that the required iterations of the inner (block) solvers must be increasing in this limit.

\begin{figure}[h!]
	\caption{Convergence study for steady-state problem, using FGMRES preconditioned by monolithic multigrid with the additive Vanka relaxation scheme using $\nu=0.4$ and $\nu=0.499$.  Left: $H^1$-seminorm error for displacement vs. mesh size. Right: $L^2$-error for pressure vs. mesh size.}
	\includegraphics{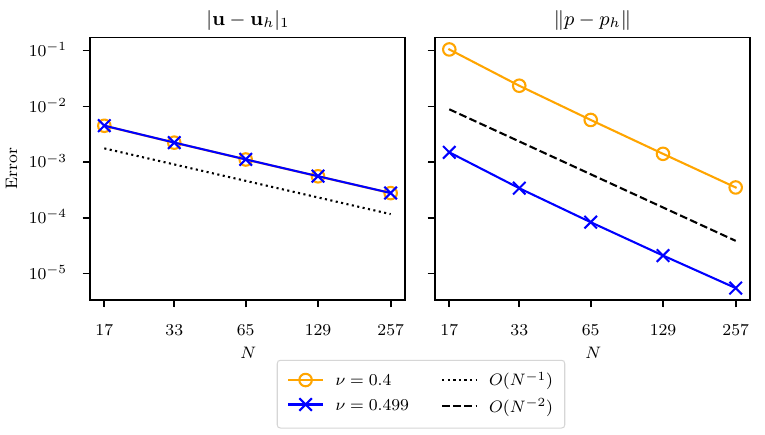}
	\label{fig:mgtime}
\end{figure}

To verify that the reduced quadrature formulation does accurately approximate the problem, we perform a convergence study of the finite-element discretization with respect to the mesh size, given by $h = 1/(N-1)$, where $N$ is the number of vertices in each dimension.
Here, we set $\tau = 1.0$ and $k = 10^{-6}$ as an example, with results shown in \cref{fig:mgtime} for $\nu=0.4$ and $\nu=0.499$.
The displacement displays a first-order convergence with respect to the $H^1$-seminorm, 
as expected, with no difference in error values for the different values of $\nu$.
The pressure displays second-order convergence despite only using P0 elements, with slight improvement as $\nu\rightarrow 0.5$.  This superconvergence is due to having a very smooth solution ($p$ is a constant) and using a uniform mesh.
Additionally, \cref{fig:mgtime_Nscaling} shows that the monolithic multigrid approach (with exact solve on a coarsest mesh with $h=1/8$) with either the
Vanka or inexact BSR relaxation methods follows the expected $O( N^2 \log(N^2))$ scaling in CPU time
with respect to problem size, even as $\nu\rightarrow 0.5$.  We note that while iterations to convergence are independent of problem size for both values of $\nu$ when using Vanka relaxation, degradation to $O(\log(N^2))$ iterations is seen for BSR relaxation as $\nu\rightarrow 0.5$.
These results further indicate the better performance of Vanka as $\nu \rightarrow 0.5$, and the loss of robustness in $\nu$ for multilevel BSR.

\begin{figure}[h!]
	\footnotesize
	\centering
	\caption{CPU time in seconds (at left) and iterations to
		convergence (at right) for FGMRES preconditioned by a monolithic multigrid V-cycle with additive Vanka and inexact BSR relaxation schemes versus mesh size on steady-state problem for $\nu=0.4$ and $\nu=0.499$.}
	\includegraphics{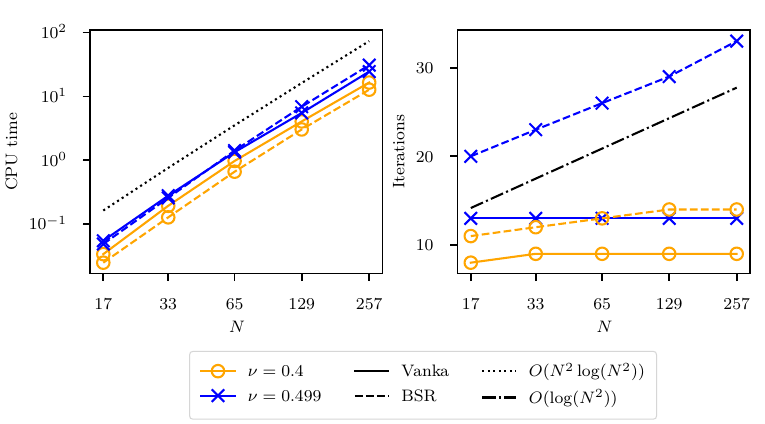}
	\label{fig:mgtime_Nscaling}
\end{figure}

\subsection{Smooth Test Problem}
We next consider a slightly more realistic test problem, now with a time-dependent smooth solution, taken from \cite{fu2019-aa}.  The manufactured solution is defined on $\Omega = [0,1]^2$, as
\begin{align*}
	\vec{u}(x,y,t) & = e^{-t}\left ( \begin{array}{c}\sin(\pi y) \left ( -\cos(\pi x) + \frac{1}{\mu + \lambda}\sin(\pi x) \right )\\
		\sin(\pi x) \left ( \cos(\pi y) + \frac{1}{\mu + \lambda}\sin(\pi y) \right )\\\end{array} \right ),\\
	p(x,y,t) & = e^{-t}\sin(\pi x)\sin(\pi y),\\
	\vec{w}(x,y,t) &=-k\nabla p,
\end{align*}
with right-hand sides chosen appropriately.  We consider Dirichlet boundary conditions on all sides for displacement and pressure.  The physical parameters are $\alpha=1$, $\mu_f = 1$, $M=10^6$, and $E= 3\times10^4$.
We perform all simulations from time $t=0$ to $t=0.5$.  Here, we use a relative residual stopping tolerance for FGMRES of $10^{-10}$, as preliminary experiments showed that this was needed to accurately resolve the pressure solution.  Moreover, we only consider the additive Vanka method, as it proved more robust in the multilevel setting.  Additionally, as we are mostly concerned with the incompressible limit, we focus on values of the Poisson ratio above 0.4.

Parameter robustness for the solver is demonstrated in Table~\ref{tab:easy_param},
showing the average solver iteration count and average CPU time over 64 time steps with time-step size, $\tau = 1/128$.
The mesh spacing is fixed to $h=1/64$, and the values of $k$ and $\nu$ are varied.
Robustness with respect to discretization parameters, $h$ and $\tau$, is shown in \cref{fig:easy_time}, for $k=10^{-6}$ and both $\nu=0.4$ and $\nu = 0.499$.  We test on meshes with $N=2^\ell+1$, for $\ell = 4$ to $8$, with $\tau = 2^{-m}$ for $m=4$ to $8$.  Here, we see nearly identical CPU times with expected $O(N^2\log(N^2))$ scaling for all values of $\tau$.  The corresponding LFA parameters from
	\cref{tab:vanka22,tab:bsr_vanka11} are used. The averaged iteration counts (not shown here) remain consistently in the range of 13 to 16 across all parameter values.
\begin{table}[h]
	\footnotesize
	\begin{center}
		\caption{Average CPU time in seconds (iterations) over 64 time steps for FGMRES preconditioned by monolithic multigrid with an additive Vanka relaxation scheme for the smooth solution problem with varying physical parameters $k$ and $\nu$.}
		\label{tab:easy_param}
		\begin{tabular}{|l | l l l l l l|}
			\hline
			\backslashbox{$\nu$}{$k$} & $1$ & $10^{-2}$ & $10^{-4}$ & $10^{-6}$ & $10^{-8}$ & $10^{-10}$\\
			\hline\hline
			$0.4$&1.536 (14.0) &1.521 (14.0) &1.543 (14.0) &1.520 (14.0) &1.679 (16.0) &1.506 (14.0) \\
			$0.45$&1.431 (13.0) &1.402 (13.0) &1.404 (13.0) &1.525 (14.0) &1.441 (14.0) &1.467 (14.0) \\
			$0.49$&1.536 (14.0) &1.518 (14.0) &1.513 (14.0) &1.513 (14.0) &1.344 (14.0) &1.294 (14.0) \\
			$0.499$&1.663 (15.0) &1.628 (15.0) &1.622 (15.0) &1.620 (15.0) &1.312 (14.0) &1.331 (15.0) \\
			\hline
		\end{tabular}
	\end{center}
\end{table}
\begin{figure}[h]
	\footnotesize
	\begin{center}
		\caption{Average CPU time in seconds for FGMRES preconditioned by monolithic multigrid for the smooth solution problem with $k=10^{-6}$ and $\nu=0.4$ (at left) and $0.499$ (at right) with varying discretization parameters, $N$ and $\tau$.}
		\includegraphics{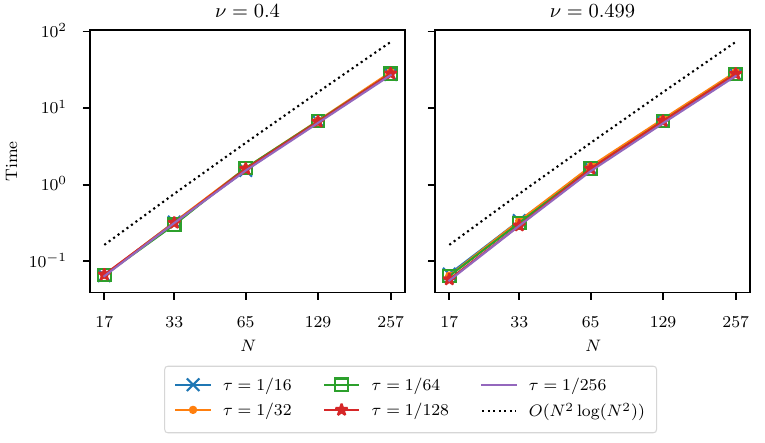}
		\label{fig:easy_time}
	\end{center}
\end{figure}

Again, to validate the discretization, we show finite-element convergence with respect to mesh size and time-step size in \cref{fig:easy_Nconv}, fixing $h=\tau$, with $k=10^{-6}$ and both $\nu=0.4$ and $\nu=0.499$.  Expected $O(h+\tau)$ convergence is seen for both the $H^1$-seminorm of $\vec{u}$ and the $L^2$ norm of $p$.

\begin{figure}[htp]
	\caption{Convergence study for smooth solution problem, using FGMRES preconditioned by monolithic multigrid with the additive Vanka relaxation scheme for $\nu=0.4$ and $\nu=0.499$ and $\tau=h=\frac{1}{N-1}$.  Left: $H^1$-seminorm error for displacement versus mesh size. Right: $L^2$-error for pressure versus mesh size.}
	\includegraphics{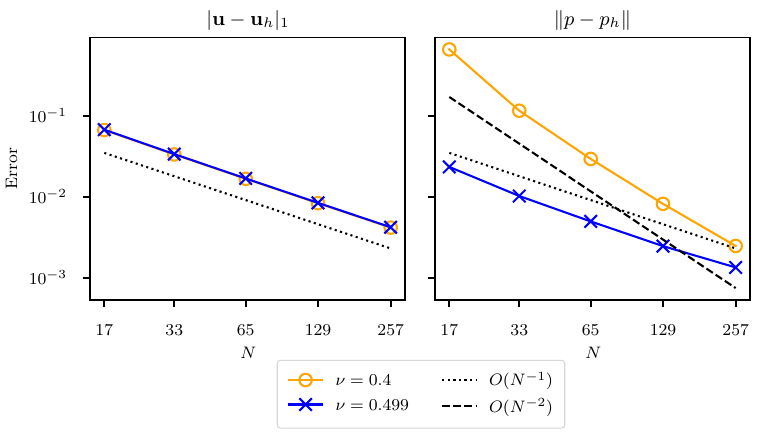}
	\label{fig:easy_Nconv}
\end{figure}

\subsection{Terzaghi's Problem} Finally, we consider a standard benchmark in poroelasticity.
The Terzaghi consolidation problem models a fluid-saturated column of a poroelastic material subject to a loading force on the top \cite{terzaghi,terzaghi2}; the cylinder height and width are $1.0$, so, once again we take $\Omega = [0,1]^2$.
The physical parameters are $\alpha=1$, $\mu_f = 1$, and $E=3 \times 10^4$,
but we take $M = \infty$ as the Biot modulus.  This means that the diagonal block of $\mathcal{A}^{RQ}$ corresponding to the pressure is zero, resulting in vertex-based Vanka blocks that are difficult to invert. To resolve this, a small positive weight of $10^{-8}$, is added to the diagonal of the Vanka blocks. This test problem has an analytical solution defined by an infinite series,
\begin{align*}
	\vec{u}(x,y,t) & = \frac{p_0}{\lambda + 2\mu}\left ( \begin{array}{c} 1 - x - \sum\limits_{i=0}^{\infty} \frac{8}{\pi^2}\frac{1}{(2i+1)^2}e^{-(2i+1)^2\pi^2k(\lambda+2\mu)t/4}\cos\left(\frac{(2i+1)\pi x}{2}\right)\\
		0\\\end{array} \right ),\\
	p(x,y,t) & = \frac{4p_0}{\pi}\sum\limits_{i=0}^{\infty}\frac{1}{(2i+1)}e^{-(2i+1)^2\pi^2k(\lambda+2\mu)t/4}\cos\left(\frac{(2i+1)\pi x}{2}\right),\\
	\vec{w}(x,y,t) &=-k\nabla p,
\end{align*}
with initial conditions, $\vec{u}(x,y,0) = \vec{0}$ and $p(x,y,0) = p_0 = 1.0$.  The problem is designed to have $\vec{0}$ as the right-hand side.

Parameter robustness for the monolithic multigrid solver is demonstrated in \cref{tab:terzaghi_paramscale},
showing the average solver FMGRES iteration count and average CPU time over 10 time steps using the additive Vanka relaxation applied to the whole system with the 20-DoF vertex-based patch (\cref{fig:patches}, right).  A relative residual stopping tolerance of $10^{-6}$ is used, with mesh spacing fixed to $h=1/64$, and the values of $k$ and $\nu$ are varied.
Due to the wide range of physical parameters considered,
there is no reasonable single time-step size for use with all parameter combinations.  Thus, we determine a parameter-dependent time scale,
$\hat{\tau} = \frac{1}{0.25 \pi^2 k (\lambda + 2\mu)}$, derived from the form of the time-dependence in the analytical solution.  All tests below simulate from time $t=0$ to $t = \hat{\tau}/10$.  Note, that this physical time-step size can vary over several orders of magnitude as we vary $k$ and $\nu$.  Thus, the optimal parameters for the steady-state model problem may not be suitable here, so these parameters were recomputed for the Terzaghi problem.  Additionally, large values of permeability are not realistic for this type of test problem, so we only consider values of $k$ in the range $10^{-4}$ to $10^{-10}$.  The results in \cref{tab:terzaghi_paramscale} highlight the robustness of the monolithic multigrid method as well as the utility of the LFA relaxation parameter predictions.

\begin{table}[h!]
	\footnotesize
	\begin{center}
		\caption{CPU time in seconds (iterations) per time step, averaged over 10 time steps with $\tau = \hat{\tau}/100$ and $h=1/64$ for FGMRES preconditioned by monolithic multigrid with additive Vanka relaxation for the Terzaghi problem with varying physical parameters $k$ and $\nu$.}
		\label{tab:terzaghi_paramscale}
		\begin{tabular}{|l | l l l l|}
			\hline\rule{0pt}{1.0\normalbaselineskip} 
			\backslashbox{$\nu$}{$k$} & $10^{-4}$ & $10^{-6}$ & $10^{-8}$ & $10^{-10}$\\
			\hline \hline
			$0.4$&0.796 (7.3) &0.803 (7.3) &0.796 (7.3) &0.802 (7.3) \\
			$0.45$&0.827 (7.6) &0.828 (7.6) &0.828 (7.6) &0.822 (7.6) \\
			$0.49$&0.903 (8.7) &0.900 (8.7) &0.898 (8.7) &0.903 (8.7) \\
			$0.499$&1.378 (14.5) &1.367 (14.5) &1.376 (14.5) &1.372 (14.5) \\
			\hline
		\end{tabular}
	\end{center}
\end{table}

In \cref{fig:terzaghi_time}, we explore the robustness with respect to the time-step size, $\tau$, and mesh size (number of points in one direction), $N$, with $h=1/(N-1)$, for $k=10^{-6}$ and both $\nu=0.4$ and $\nu=0.499$, for monolithic multigrid using Vanka relaxation  as described above.  Note that with a smaller time-step size, more time steps are needed to get to the same final time.  In all cases, the iteration counts remain stable (no worse than $O(\ln(N^2))$), and the computational time scales as $O(N^2\log(N^2))$.

\begin{figure}[h!]
	\footnotesize
	\begin{center}
		\caption{Average CPU time in seconds (top) and iterations to convergence (bottom) per time-step over time interval $[0,\hat{\tau}/10]$ for FGMRES preconditioned by monolithic multigrid with additive Vanka relaxation for the Terzaghi problem. Here, $k=10^{-6}$ and $\nu=0.4$ (left) and $0.499$ (right).}
		\includegraphics{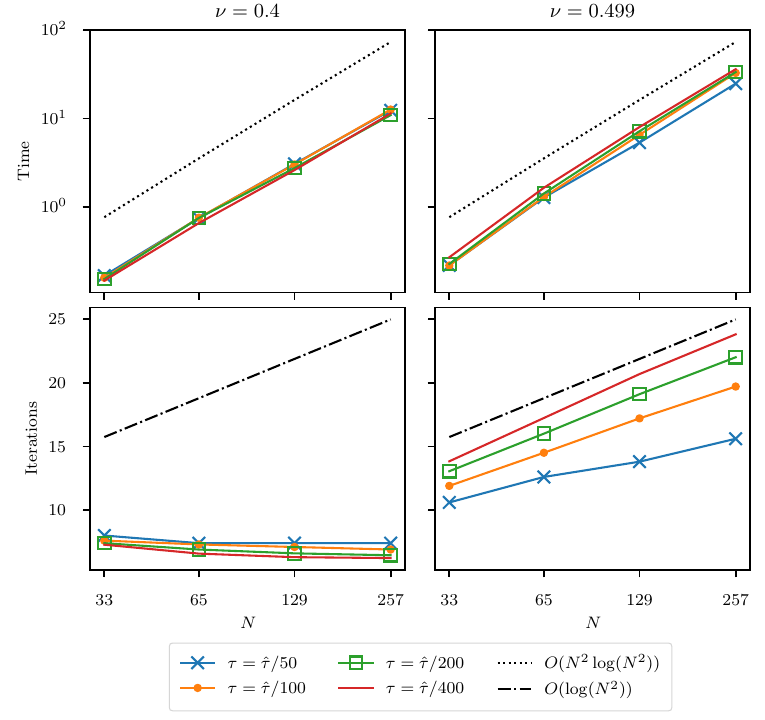}
		\label{fig:terzaghi_time}
	\end{center}
\end{figure}

\section{Conclusions}\label{sec:Summary}
In this paper, we investigate the construction of parameter-robust preconditioners for three-field models of Biot poroelasticity.  Following \cite{rodrigo2018new, adler2019robust}, we consider a bubble-enriched P1-RT0-P0 finite-element discretization; however, in order to allow for robust solvers, we introduce a reduced quadrature approximation and show that the discretization quality does not suffer from this change.  With this, and suitable treatment of divergence-free displacements in both the relaxation and interpolation operators, we derive robust monolithic multigrid methods to solve this problem, with both Vanka and inexact Braess-Sarazin relaxation schemes.  In numerical tests, we see that the additive form of Vanka relaxation is more robust than inexact Braess-Sarazin.  Both approaches outperform the block-triangular preconditioner of \cite{adler2019robust}, particularly in the incompressible limit. Improving robustness of inexact Braess-Sarazin relaxation in the small permeability and nearly incompressible limits is an interesting question for future work.

Another natural topic for future work is extending the preconditioners developed here for the ``bubble-eliminated'' system described in \cite{rodrigo2018new, adler2019robust}, where an approximate Schur complement is used to remove the face-based displacement DoFs.  Additionally, more complicated
models of poroelasticity will be considered, including their implementation for three-dimensional models, and for nonlinear models that describe porous materials with fractures, see~\cite{BudisaHu2019,FlemischBerreBoonFumagalliSchwenckScottiStefanssonTatomir2018,FlemischFumagalliScotti2016,NordbottenBoonFumagalliKeilegavlen2018} and references therein.  Developing robust multigrid solvers for the linearizations of these systems will aid in the development of fast simulations for real-world problems in the geosciences and biomedical research.
 
 \appendix
 
 \section{LFA for the Reduced-Quadrature Discretization}\label{secSM:LFA-RQD}
 Consider the discretization matrix, $\mathcal{A}^{\text{RQ}}$, from \cref{block_form_rq},
 \begin{equation*}
 	\mathcal{A}^{\text{RQ}}= \left(
 	\begin{array}{ccc}
 		A_{\bm{u}}^{\text{RQ}} &0                      & \alpha B_{\bm{u}}^{\top} \\
 		0                      & \tau M_{\bm w}       & \tau B_{\bm w}^{\top} \\
 		\alpha B_{\bm u}      & \tau B_{\bm w}       & -\frac{1}{M} M_p
 	\end{array}
 	\right).
 \end{equation*}
 As discussed in \Cref{subsec:LFA-Biot}, the calculation of the block symbol of $\mathcal{A}^{\text{RQ}}$ is complicated because of both the different discretization spaces used for $\bm{u}$, $\bm{w}$, and $p$ and the different basis functions used within each of these spaces.  By rewriting $\mathcal{A}^{\text{RQ}}$ in $10\times 10$ block form, we can expose Toeplitz structure within each block, identifying each block in the system with one ``type'' of basis function used in the discretization.  The same approach was used, for example, in \cite{HMFEMStokes,HM2018LFALaplace} to define LFA representations of similarly structured finite-element discretizations of the Laplacian and Stokes operators.  In all that follows, we consider a uniform mesh of the unit square domain, constructed by partitioning the domain into square elements that are each then cut once diagonally (from top left to bottom right) to form a triangulation of the domain.
 
 We first consider the diagonal displacement operator, $A_{\bm{u}}^{\text{RQ}}$, noting that there are five distinct types of basis functions used for $\bm{u}$ in the discretization, leading to $5\times 5$ block structure of its LFA symbol.  These basis functions are the two P1 components of the displacement, along with the three face-based DoFs for the bubble functions.  To give the LFA representation of this operator, we first write its stencil in terms of these basis functions, then use the techniques of  \cite{HMFEMStokes,HM2018LFALaplace} to compute the Fourier symbols.
 
 Recall that we can separate the reduced-quadrature displacement operator into two terms,
 \begin{equation*}
 	A^{\text{RQ}}_{\bm{u}}=2\mu A_{\varepsilon} + \lambda B_{\bm{u}}^{\top} M_p^{-1} B_{\bm{u}},
 \end{equation*}
 where $2\mu A_{\varepsilon}$ corresponds to the weak form, $2\mu\left(
 \varepsilon(\bm{u}),\varepsilon(\bm{v})\right)$, and the second term is the
 reduced-quadrature discretization of the grad-div operator.  We
 separately compute LFA symbols for each of these terms, noting that we
 can write the symbol for the second term as a product of symbols for
 its component parts (see, for example,
 \cite{KKahl_NKintscher_2020a}), which are needed elsewhere in the
 symbol for $\mathcal{A}^{\text{RQ}}$.
 
 We begin by considering the stencil for $2\mu A_\varepsilon$ in three pieces.
 \cref{bubble-stencil-plot} shows two stencils for $2\mu A_{\varepsilon}$, corresponding to the face-based displacement DoFs along the diagonal edges (left) and the horizontal edges (right).  This figure shows only connections between the bubble DoFs. Connections between bubble and P1 DoFs are discussed below.  The stencil for the vertical edges is obtained by a rotation and reflection of that shown for horizontal edges.
 The stencils for the connections between P1 DoFs of the same type naturally have a five-point structure due to symmetry.  For the P1 $x$-component of displacement, the stencil is
 \begin{equation*}
 	\begin{bmatrix}
 		&     -\mu    &         \\
 		-2\mu&     6\mu&      -2\mu   \\
 		&     -\mu      &
 	\end{bmatrix},
 \end{equation*}
 with a $90^\circ$ rotation for the P1 $y$-component.
 The stencil between the two P1 components of the displacement is given by
 \begin{equation*}
 	\begin{bmatrix}
 		\frac{\mu}{2} &     -\frac{\mu}{2}    &         \\
 		-\frac{\mu}{2}&     \mu&      -\frac{\mu}{2}  \\
 		&    -\frac{\mu}{2}   &    \frac{\mu}{2}
 	\end{bmatrix}.
 \end{equation*}
 Finally, connections between the P1 $x$-component of the displacement and the bubble DoFs are shown at left of \cref{fig:P1-bubble-stencil}, while those between the P1 $y$-component of the displacement and the bubble DoFs are shown at right.  Connections between the bubble DoFs and the P1 components of the displacement are transposes of these connections.
 \begin{figure}[htp]
 	\centering
 	\begin{tikzpicture}[font=\large]
 		\draw [-] [thick] (0,0) --(4,0);
 		\draw [-] [thick] (0,0) --(0,4);
 		\draw [-] [thick] (4,0) --(4,4);
 		\draw [-] [thick] (0,4) --(4,4);
 		\draw [-] [thick] (0,4) --(4,0);
 		\node at (2,0) {$\bullet$};  \node at (2,0.5) {\large{$\frac{-5\sqrt{2}\mu}{3}$}};
 		\node at (2,4) {$\bullet$};  \node at (2,4.5) {\large{$\frac{-5\sqrt{2}\mu}{3}$}};
 		
 		\node at (0,2) {$\bullet$};  \node at (-0.65,2) {\large{$\frac{-5\sqrt{2}\mu}{3}$}};
 		\node at (4,2) {$\bullet$};  \node at (4.65,2.) {\large{$\frac{-5\sqrt{2}\mu}{3}$}};
 		\node at (2,2) {$\bullet$};   \node at (2,2.5) {\large{$\frac{28\mu}{3}$}};
 	\end{tikzpicture}
 	\hspace{4mm}
 	\begin{tikzpicture}
 		\draw [-] [thick] (0,2) --(4,2);
 		\draw [-] [thick] (0,2) --(0,4);
 		\draw [-] [thick] (0,4) --(4,2);
 		\draw [-] [thick] (4,0) --(4,2);
 		\draw [-] [thick] (0,2) --(4,0);
 		
 		\node at (2,2) {$\bullet$};  \node at (2,1.7) {{$8\mu$}};
 		\node at (0,3) {$\bullet$};  \node at (0.25,3) {\large{$\frac{2\mu}{3}$}};
 		
 		\node at (4,1) {$\bullet$};  \node at (4.25,1) {\large{$\frac{2\mu}{3}$}};
 		\node at (2,1) {$\bullet$};   \node at (1.8,0.6) {\large{$\frac{-5\sqrt{2}\mu}{3}$}};
 		\node at (2,3) {$\bullet$};  \node at (2.2,3.3) {\large{$\frac{-5\sqrt{2}\mu}{3}$}};
 		
 	\end{tikzpicture}
 	\caption{Stencils for the bubble DoFs in $2\mu A_\varepsilon$. Left: stencil associated with the diagonal edges. Right: stencil associated with the horizontal edges.}
 	\label{bubble-stencil-plot}
 \end{figure}
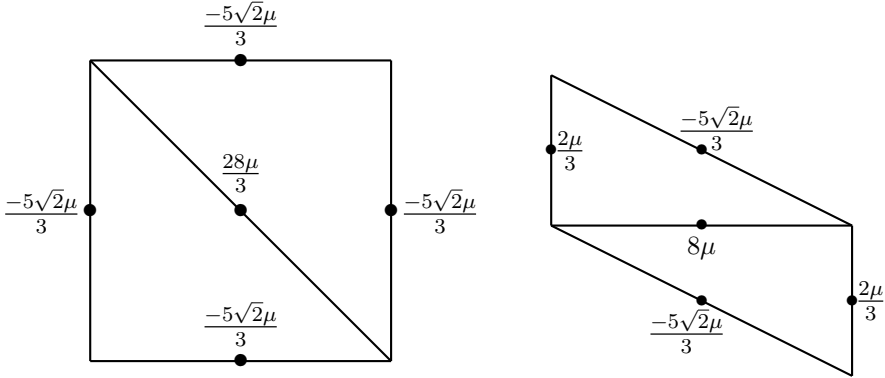
 \begin{figure}[htp]
 	\centering
 	\begin{tikzpicture}[font=\normalsize]
 		\draw [-] [thick] (2,0) --(4,0); \draw [-] [thick] (0,4) --(2,4);
 		\draw [-] [thick] (0,2) --(0,4); \draw [-] [thick] (4,0) --(4,2);
 		\draw [-] [thick] (0,2) --(2,0);  \draw [-] [thick] (0,4) --(4,0);
 		\draw [-] [thick] (2,4) --(4,2); \draw [-] [thick] (0,2) --(4,2);
 		\draw [-] [thick] (2,0) --(2,4);
 		
 		\node at (1,3) {$\bullet$};  \node at (1.45,3.1) {\large{$\frac{4\sqrt{2}\mu}{3}$}};
 		\node at (3,3) {$\bullet$};  \node at (3.6,3) {\large{$\frac{-4\sqrt{2}\mu}{3}$}};
 		\node at (1,1) {$\bullet$};  \node at (0.3,1.) {\large{$\frac{-4\sqrt{2}\mu}{3}$}};
 		\node at (3,1) {$\bullet$};  \node at (3.45,1.1) {\large{$\frac{4\sqrt{2}\mu}{3}$}};
 		
 		\node at (0,3) {$\bullet$};  \node at (-0.5,3) {\large{$\frac{-4\mu}{3}$}};
 		\node at (2,3) {$\bullet$};  \node at (2.3,3) {\large{$\frac{4\mu}{3}$}};
 		\node at (2,1) {$\bullet$};  \node at (2.3,1) {\large{$\frac{4\mu}{3}$}};
 		\node at (4,1) {$\bullet$};  \node at (4.4,1) {\large{$\frac{-4\mu}{3}$}};
 	\end{tikzpicture}
 	\hspace{4mm}
 	\begin{tikzpicture}[font=\normalsize]
 		\draw [-] [thick] (2,0) --(4,0); \draw [-] [thick] (0,4) --(2,4);
 		\draw [-] [thick] (0,2) --(0,4); \draw [-] [thick] (4,0) --(4,2);
 		\draw [-] [thick] (0,2) --(2,0);  \draw [-] [thick] (0,4) --(4,0);
 		\draw [-] [thick] (2,4) --(4,2); \draw [-] [thick] (0,2) --(4,2);
 		\draw [-] [thick] (2,0) --(2,4);
 		
 		\node at (1,3) {$\bullet$};  \node at (1.4,3.) {\large{$\frac{4\sqrt{2}\mu}{3}$}};
 		\node at (3,3) {$\bullet$};  \node at (3.6,3) {\large{$\frac{-4\sqrt{2}\mu}{3}$}};
 		\node at (1,1) {$\bullet$};  \node at (0.3,1) {\large{$\frac{-4\sqrt{2}\mu}{3}$}};
 		\node at (3,1) {$\bullet$};  \node at (3.5,1) {\large{$\frac{4\sqrt{2}\mu}{3}$}};
 		
 		\node at (1,2) {$\bullet$};  \node at (1,2.4) {\large{$\frac{4\mu}{3}$}};
 		\node at (1,4) {$\bullet$};  \node at (1,3.65) {\large{$\frac{-4\mu}{3}$}};
 		\node at (3,0) {$\bullet$};  \node at (3,0.4) {\large{$\frac{-4\mu}{3}$}};
 		\node at (3,2) {$\bullet$};  \node at (3,2.4) {\large{$\frac{4\mu}{3}$}};
 	\end{tikzpicture}
 	\caption{Connections between the P1 DoFs and the bubble DoFs in $2\mu A_\varepsilon$.  At left, the P1 $x$-component of the displacement stencil to bubble DoFs.  At right, the P1 $y$-component of the displacement stencil to bubble DoFs.}
 	\label{fig:P1-bubble-stencil}
 \end{figure}
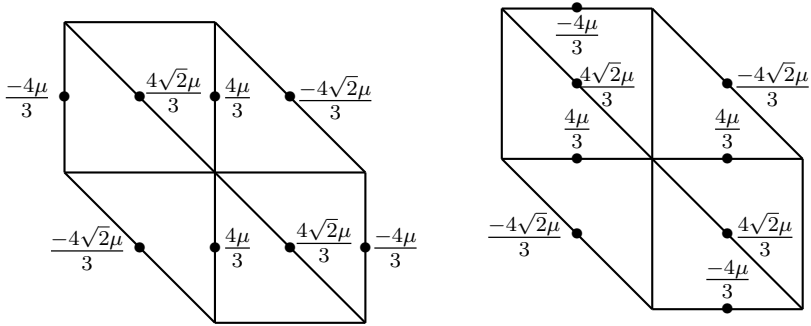

 Symbols for these pieces of $2\mu A_\varepsilon$ are then assembled using standard techniques.  The diagonal components of the symbol are directly calculated using \cref{formulation-symbol-classical} (trivially so for the bubble DoFs, where the diagonal blocks are themselves diagonal matrices).  For the off-diagonal entries, proper treatment of the non-collocated nature of the DoFs in the finite-element discretization is necessary \cite{HMFEMStokes, HM2018LFALaplace}.  Here, we base the Fourier symbols on the offset in DoF positions on the mesh; that is, when we consider the prototypical Fourier basis functions, $\varphi(\bm{\theta},\bm{x})= e^{\iota\bm{\theta}\cdot\bm{x}/{h}}$, we note that the position on the mesh, $\bm{x}$, plays an important role in the definition of the basis.  When considering two different types of DoFs, located at different positions on the mesh, the classical symbol definition hides the fact that we may use different sets of DoF locations for the domain and range of an off-diagonal block.  In essence, this comes down to the set $\bm{S}$ in \cref{formulation-symbol-classical}.  If we consider operator $L_h$ to be the off-diagonal block in the block-row of $A_\varepsilon$ corresponding to DoF-type 1 and the block column corresponding to DoF-type 2, we have
 \[
 \left(L_h\varphi(\bm{\theta},\cdot)\right)(\bm{x}_1) = \sum_{\bm{\kappa}\in\bm{S}}s_{\bm{\kappa}}e^{\iota \bm{\theta}\cdot(\bm{x}_1 + \bm{\kappa}h)/h}.
 \]
 For the right-hand side to be well-defined, we need $\bm{x}_1 + \bm{\kappa}h$ to correspond to a point, $\bm{x}_2$, on the mesh of DoF-type 2.  This necessarily changes the set $\bm{S}$ from being a subset of $\mathbb{Z}^2$ to being one that accounts for the offset between the two DoF types, accounting for fractional $h$ values in $\bm{\kappa}$.  Here, we identify the horizontal-edge face bubble DoF as having offset $(h/2,0)$ from the nodal P1 DoFs, the vertical-edge face bubble DoFs as having offset $(0,h/2)$ from the nodes, and the diagonal-edge face bubble DoFs as having offset $(h/2,h/2)$ from the nodes.  Accounting for these offsets gives the symbol for the bubble-bubble DoF connections (ordered as diagonal, horizontal, and vertical edges),
 \[
 \mu\begin{pmatrix}
 	\frac{28}{3}&  -\frac{10\sqrt{2}}{3}\cos(\frac{\theta_1}{2})&    -\frac{10\sqrt{2}}{3}\cos(\frac{\theta_2}{2})\\
 	-\frac{10\sqrt{2}}{3}\cos(\frac{\theta_1}{2})&    8&  \frac{4}{3}\cos(\frac{\theta_1-\theta_2}{2})&\\
 	-\frac{10\sqrt{2}}{3}\cos(\frac{\theta_2}{2})&   \frac{4}{3}\cos(\frac{\theta_1-\theta_2}{2})&  8
 \end{pmatrix},
 \]
 the symbol for the P1-P1 DoF connections,
 \[
 \mu\begin{pmatrix}
 	6-4\cos(\theta_1)-2\cos(\theta_2)& 1-\cos(\theta_1)-\cos(\theta_2)+\cos(\theta_1-\theta_2)\\
 	1-\cos(\theta_1)-\cos(\theta_2)+\cos(\theta_1-\theta_2)&  6-4\cos(\theta_2)-2\cos(\theta_1)
 \end{pmatrix},
 \]
 and the symbol for the contributions from the bubble DoFs to the P1 DoFs,
 \[
 \frac{8\mu}{3}\hspace{-1mm}\begin{pmatrix}
 	\sqrt{2}\big(\cos(\frac{\theta_1-\theta_2}{2})-\cos(\frac{\theta_1+\theta_2}{2})\big)& \cos(\frac{\theta_2}{2})\hspace{-0.8mm}-\hspace{-0.8mm}\cos(\theta_1-\frac{\theta_2}{2})& 0\\
 	\sqrt{2}\big(\cos(\frac{\theta_1-\theta_2}{2})-\cos(\frac{\theta_1+\theta_2}{2})\big)& 0 & \cos(\frac{\theta_1}{2})-\cos(\theta_2-\frac{\theta_1}{2})
 \end{pmatrix},
 \]
 with a transpose of this symbol for contributions from P1 DoFs to bubble DoFs.

 Similar calculations follow for the stencils and symbols of the other terms in $\mathcal{A}^{\text{RQ}}$.  For $M_{\bm{w}}$, we make use of the same adjustments to account for the staggering of the face-based RT0 DoFs, leading to the symbol,
 \begin{equation*}
 	\widetilde{ M}_{\boldsymbol{\omega}}(\theta_1,\theta_2) =\frac{\mu_fh^2}{3k}\begin{pmatrix}
 		2&    0&   0\\
 		0&     2&  -\cos(\frac{\theta_1-\theta_2}{2})\\
 		0&  -\cos(\frac{\theta_1-\theta_2}{2})&  2
 	\end{pmatrix}.
 \end{equation*}
 For the P0 discretization of pressure, we have 2 types of DoFs, associated with the lower-left and upper-right triangles when the quadrilateral mesh is cut into triangles.  Since the mass matrix is diagonal, we have the symbol,
 \begin{equation*}
 	\widetilde{ M}_p(\theta_1,\theta_2) = \frac{h^2}{2}\begin{pmatrix} 1 &0 \\0 & 1 \end{pmatrix}.
 \end{equation*}
 We then write $B_{\bm u}$ as a $2\times 5$ system of operators with symbol,
 \begin{equation*}
 	h \begin{pmatrix}
 		-\frac{2\sqrt{2}}{3}&     \frac{2}{3}e^{\frac{-\iota\theta_1}{2}}& \frac{2}{3}e^{\frac{-\iota\theta_2}{2}} &   \frac{-1}{2}\left(e^{\frac{\iota\theta_1}{2}}-e^{\frac{-\iota\theta_1}{2}}\right)e^{\frac{-\iota\theta_2}{2}}&   \frac{-1}{2}\left(e^{\frac{\iota\theta_2}{2}}-e^{\frac{-\iota\theta_2}{2}}\right)e^{\frac{-\iota\theta_1}{2}} \\
 		\frac{2\sqrt{2}}{3}&     -\frac{2}{3}e^{\frac{\iota\theta_1}{2}}& -\frac{2}{3}e^{\frac{\iota\theta_2}{2}} & \frac{-1}{2}\left(e^{\frac{\iota\theta_1}{2}}-e^{\frac{-\iota\theta_1}{2}}\right)e^{\frac{\iota\theta_2}{2}}&   \frac{-1}{2}\left(e^{\frac{\iota\theta_2}{2}}-e^{\frac{-\iota\theta_2}{2}}\right)e^{\frac{\iota\theta_1}{2}}
 	\end{pmatrix}.
 \end{equation*}
 Similarly, ${B}_{\bm w}$ is a $2\times 3$ block operator with symbol,
 \begin{equation*}
 	\widetilde{B}_{\bm w}(\theta_1,\theta_2)= h\begin{pmatrix}
 		-\sqrt{2}&     e^{-\frac{\iota\theta_1}{2}}&  e^{-\frac{\iota\theta_2}{2}}\\
 		\sqrt{2}&      -e^{\frac{\iota\theta_1}{2}}&  -e^{\frac{\iota\theta_1}{2}}
 	\end{pmatrix}.
 \end{equation*}
 From the symbols for $M_p$ and $B_{\bm u}$, we can compute the rest of the symbol for $A_{\bm u}^{\text{RQ}}$.  Taking transposes for the off-diagonal connections gives the rest of the $10\times 10$ block symbol of $\mathcal{A}^{\text{RQ}}$.
 %
 \section{LFA Representation of Grid-transfer Operators}\label{secSM:LFA-transfer-operator}
 As discussed in \Cref{sec:mmg}, we use the standard finite-element interpolation operators for $\bm{w}$ and $p$ and the modified (divergence-preserving) interpolation operator for $\bm{u}$ (see \Cref{sec:Pdiv}).  We use their transposes for restriction.  We compute symbols for the restriction operators, with those for interpolation determined as the scaled transposes, $\widetilde{P}(\bm{\theta}) = \frac{1}{4}\widetilde{R}(\bm{\theta})^\top$, in the standard way for finite-element discretizations \cite{MR1807961}.  As above, the calculation of these symbols is complicated by the staggered locations of the finite-element DoFs.
 
 Consider an arbitrary restriction operator for a scalar function (e.g., discretized in P1) characterized by a constant coefficient
 stencil, $R_h\overset{\wedge}{=} [r_{\bm{\kappa}}]$. Then, an infinite grid function $w_{h}: \mathbf{G}_h \rightarrow \mathbb{R}$ (or $\mathbb{C}$) is
 transferred to the coarse grid, $\mathbf{ G}_{2h}$, as
 \begin{eqnarray*}
 	( R_h w_{h})(\bm{ x})&=&\sum_{\kappa\in{W}}r_{\bm{\kappa}}w_{h}(\bm{ x}+\bm{\kappa} h),\,\,\bm{ x}\in \mathbf{G}_{2h},
 \end{eqnarray*}
 where $W$ is a finite subset of $\mathbb{Z}^2$ describing the stencil $[r_{\bm{\kappa}}]_h$.
 
 Given a low-frequency $\bm{\theta}^{(0,0)}$ with harmonic modes $\bm{\theta^{\alpha}}$ and taking $w_h$ to be the Fourier mode, $\varphi(\bm{\theta^{\alpha}},\bm{x} )=e^{\iota\bm{\theta^{\alpha}}\cdot\bm{x}/h}$, we have
 \begin{equation*}
 	( R_h \varphi(\bm{\theta^{\alpha}},\cdot ))(\bm{ x})= \left(\sum_{\bm{\kappa}\in W} r_{\bm{\kappa}}e^{\iota \bm{\kappa}\cdot \bm{\theta}^{\bm{\alpha}}}\right)\varphi_{2h}(2\bm{\theta}^{(0,0)},\boldsymbol{x}),\,\,\bm{ x}\in \mathbf{G}_{2h}.
 \end{equation*}
 \begin{definition}\label{def-R-symbol-classical}
 	We call $\widetilde{R}_h(\boldsymbol{\theta^\alpha})=\displaystyle\sum_{\bm{\kappa}\in W} r_{\bm{\kappa}}e^{\iota \bm{\kappa}\cdot \bm{\theta}^{\bm{\alpha}}}$ the restriction symbol of $R_h$.
 \end{definition}

 For staggered meshes, we again generalize the classical restriction symbol to allow restriction from one type of DoF to another.  Following \cite{HM2018LFALaplace}, we give the general form of the Fourier representation of a restriction operator as follows,
 \begin{definition}\label{def-R-symbol-G}
 	Let $\bm{x}$ be a DoF location on grid $\bm{G}_{2h}$ to which $R_h$ restricts, and let $W$ be the set of offsets on grid $\bm{G}_h$ from which we restrict to $\bm{x}$.   We call $\widetilde{R}_h(\boldsymbol{\theta^\alpha})=\displaystyle\left(\sum_{\bm{\kappa}\in W} r_{\bm{\kappa}}e^{\iota \bm{\kappa}\cdot \bm{\theta}^{\bm{\alpha}}}\right)e^{\iota \pi \bm{\alpha}\cdot \bm{x}/h }$ the restriction symbol of $R_h$.
 \end{definition}
 Note that $\widetilde{R}_h(\boldsymbol{\theta^\alpha})$ is independent of the particular $\bm{G}_{2h}$ point, $\bm{x}$, used to define the restriction symbol in \cref{def-R-symbol-G}, since all points on $\bm{G}_{2h}$ differ by integer multiples of $2h$.
 
 Recall from \cref{eq:block_interp} that we consider a block-structured restriction operator,
 \begin{equation*}
 	R = \begin{pmatrix}
 		R_{\bm u} & 0 & 0 \\
 		0 & R_{\bm w} & 0 \\
 		0 & 0 & R_p \end{pmatrix},
 \end{equation*}
 where $R_{\bm u},  R_{\bm w}, R_{ p}$ are $5\times 5$, $3\times 3$, and $2\times 2$ block-structured systems of operators, respectively.  As a result, their symbols are $5\times 5$, $3\times 3$, and $2\times 2$ matrices, determined by the coefficients in the restriction stencils.  Here, we do not give the stencils for these operators, just their symbols.
 
 The symbol for $R_{\bm u}$ can be computed in 4 parts.
 For convenience, we take $\eta_1=(-1)^{\alpha_1}$ and $\eta_2=(-1)^{\alpha_2}$.  The $3\times 3$ sub-block of $\widetilde{R}_{\bm u}(\bm{\theta^\alpha})$, corresponding to the bubble DoFs (in the same ordering as above) is
 \[
 \begin{pmatrix}
 	\frac{1}{4}(e^{\frac{\iota(-\theta_1^{\alpha_1}+\theta_2^{\alpha_2})}{2}}\hspace{-1mm}+\hspace{-0.75mm}e^{\frac{\iota(\theta_1^{\alpha_1}-\theta_2^{\alpha_2})}{2}})\eta_1\eta_2
 	&\hspace{-2mm}\frac{\sqrt{2}}{8}(e^{\frac{\iota \theta_2^{\alpha_2}}{2}}\hspace{-1mm}+\hspace{-0.75mm}e^{\frac{-\iota \theta_2^{\alpha_2}}{2}})\eta_1\eta_2
 	&\hspace{-2mm}\frac{\sqrt{2}}{8}(e^{\frac{\iota \theta_1^{\alpha_1}}{2}}\hspace{-1mm}+\hspace{-0.75mm}e^{\frac{-\iota \theta_1^{\alpha_1}}{2}})\eta_1\eta_2
 	\\
 	\frac{5\sqrt{2}}{8}(e^{\frac{\iota(-\theta_1^{\alpha_1}+\theta_2^{\alpha_2})}{2}}\hspace{-1mm}+\hspace{-0.75mm}e^{\frac{\iota(\theta_1^{\alpha_1}-\theta_2^{\alpha_2})}{2}})\eta_2
 	&   \frac{1}{4}(e^{\frac{\iota \theta_2^{\alpha_2}}{2}}\hspace{-1mm}+\hspace{-0.75mm}e^{\frac{-\iota \theta_2^{\alpha_2}}{2}})\eta_2
 	&  -(e^{\frac{\iota \theta_1^{\alpha_1}}{2}}\hspace{-1mm}+\hspace{-0.75mm}e^{\frac{-\iota \theta_1^{\alpha_1}}{2}})\eta_2
 	\\
 	\frac{5\sqrt{2}}{8}(e^{\frac{\iota(-\theta_1^{\alpha_1}+\theta_2^{\alpha_2})}{2}}\hspace{-1mm}+\hspace{-0.75mm}e^{\frac{\iota(\theta_1^{\alpha_1}-\theta_2^{\alpha_2})}{2}})\eta_1
 	&  -(e^{\frac{\iota \theta_2^{\alpha_2}}{2}}\hspace{-1mm}+\hspace{-0.75mm}e^{\frac{-\iota \theta_2^{\alpha_2}}{2}})\eta_1
 	& \frac{1}{4}(e^{\frac{\iota \theta_1^{\alpha_1}}{2}}\hspace{-1mm}+\hspace{-0.75mm}e^{\frac{-\iota \theta_1^{\alpha_1}}{2}})\eta_1
 \end{pmatrix}.
 \]
 The $2\times 2$ submatrix of $\widetilde{R}_{\bm u}(\bm{\theta^\alpha})$ corresponding to the P1 components of the displacement is diagonal, with entry
 $$1+\frac{1}{2}\big((e^{\iota \theta_1^{\alpha_1}}+e^{-\iota \theta_1^{\alpha_1}})+(e^{\iota \theta_2^{\alpha_2}}+e^{-\iota \theta_2^{\alpha_2}})+ (e^{\iota \theta_1^{\alpha_1}}e^{-\iota \theta_2^{\alpha_2}}+e^{-\iota \theta_1^{\alpha_1}}e^{\iota \theta_2^{\alpha_2}})\big),$$ for both components.  The contributions to the symbol from the P1 DoFs to the bubble DoFs are given by
 \[
 \begin{pmatrix}  \frac{\sqrt{2}}{2}\eta_1\eta_2  & \frac{\sqrt{2}}{2}\eta_1\eta_2 \\
 	\eta_2  &  0 \\
 	0 &  \eta_1
 \end{pmatrix},
 \]
 while those from the bubble DoFs to the P1 DoFs are given by
 \[
 \begin{pmatrix}
 	d_1 & d_2 & d_3\\
 	d_4 & d_5 & d_6
 \end{pmatrix},
 \]
 with,
 \begin{eqnarray*}
 	d_1&=& \frac{3\sqrt{2}}{16}\big(e^{\frac{\iota(\theta_1^{\alpha_1}+\theta_2^{\alpha_2})}{2}}+e^{\frac{-\iota(\theta_1^{\alpha_1}+\theta_2^{\alpha_2})}{2}}\big)
 	-\frac{3\sqrt{2}}{16}\big(e^{\frac{\iota(3\theta_1^{\alpha_1}-\theta_2^{\alpha_2})}{2}}+e^{\frac{\iota(-3\theta_1^{\alpha_1}+\theta_2^{\alpha_2})}{2}}\big),
 	\\
 	d_2&=&  -\frac{3}{8}\big(e^{\frac{\iota (2\theta_1^{\alpha_1}+\theta_2^{\alpha_2})}{2}}+e^{\frac{-\iota (2\theta_1^{\alpha_1}+\theta_2^{\alpha_2})}{2}}\big)
 	+ \frac{3}{8}\big(e^{\frac{\iota (-2\theta_1^{\alpha_1}+\theta_2^{\alpha_2})}{2}}+e^{\frac{\iota (2\theta_1^{\alpha_1}-\theta_2^{\alpha_2})}{2}}\big),
 	\\
 	d_3&=&  -\frac{3}{8}\big(e^{\frac{\iota (\theta_1^{\alpha_1}+2\theta_2^{\alpha_2})}{2}}+e^{\frac{-\iota (\theta_1^{\alpha_1}+2\theta_2^{\alpha_2})}{2}}\big)
 	+ \frac{3}{8}\big(e^{\frac{\iota (3\theta_1^{\alpha_1}-2\theta_2^{\alpha_2})}{2}}+e^{\frac{\iota (-3\theta_1^{\alpha_1}+2\theta_2^{\alpha_2})}{2}}\big),
 	\\
 	d_4& =& \frac{3\sqrt{2}}{16}\big(e^{\frac{\iota (\theta_1^{\alpha_1}+\theta_2^{\alpha_2})}{2}}+e^{\frac{-\iota (\theta_1^{\alpha_1}+\theta_2^{\alpha_2})}{2}}\big)
 	-\frac{3\sqrt{2}}{16}\big(e^{\frac{\iota (\theta_1^{\alpha_1}-3\theta_2^{\alpha_2})}{2}}+e^{\frac{\iota (-\theta_1^{\alpha_1}+3\theta_2^{\alpha_2})}{2}}\big),
 	\\
 	d_5&= &  -\frac{3}{8}\big(e^{\frac{\iota (2\theta_1^{\alpha_1}+\theta_2^{\alpha_2})}{2}}+e^{\frac{-\iota (2\theta_1^{\alpha_1}+\theta_2^{\alpha_2})}{2}}\big)
 	+ \frac{3}{8}\big(e^{\frac{\iota (-2\theta_1^{\alpha_1}+3\theta_2^{\alpha_2})}{2}}+e^{\frac{\iota (2\theta_1^{\alpha_1}-3\theta_2^{\alpha_2})}{2}}\big),
 	\\
 	d_6& =&  -\frac{3}{8}\big(e^{\frac{\iota (\theta_1^{\alpha_1}+2\theta_2^{\alpha_2})}{2}}+e^{\frac{-\iota (\theta_1^{\alpha_1}+2\theta_2^{\alpha_2})}{2}}\big)
 	+ \frac{3}{8}\big(e^{\frac{\iota (\theta_1^{\alpha_1}-2\theta_2^{\alpha_2})}{2}}+e^{\frac{\iota (-\theta_1^{\alpha_1}+2\theta_2^{\alpha_2})}{2}}\big).
 \end{eqnarray*}
 The symbol for $R_{\bm w}$ is given by
 \begin{equation*}
 	\begin{pmatrix}
 		(2c_3 + c_4)\eta_1\eta_2&\sqrt{2}c_2\eta_1\eta_2&\sqrt{2}c_1\eta_1\eta_2\\
 		\frac{\sqrt{2}}{2}c_3\eta_2&(2c_2 + c_5)\eta_2&-c_1\eta_2\\
 		\frac{\sqrt{2}}{2}c_3\eta_1&-c_2\eta_1&(2c_1 + c_6)\eta_1\\
 	\end{pmatrix},
 \end{equation*}
 where
 \begin{align*}
 	c_1&=\cos\left (\frac{\theta_1^{\alpha_1}}{2}\right ), &c_2&=\cos\left (\frac{\theta_2^{\alpha_2}}{2}\right ),\\
 	c_3&=\cos\left (\frac{\theta_1^{\alpha_1}-\theta_2^{\alpha_2}}{2}\right ),&c_4&=\cos\left (\frac{\theta_1^{\alpha_1}+\theta_2^{\alpha_2}}{2}\right ),\\
 	c_5&=\cos\left (\frac{2\theta_1^{\alpha_1}-\theta_2^{\alpha_2}}{2}\right ),&c_6&=\cos\left (\frac{\theta_1^{\alpha_1}-2\theta_2^{\alpha_2}}{2}\right ).\\
 \end{align*}
 
 %
 Finally, the symbol for $R_{p}$ is
 \begin{equation*}
 	\eta_1\eta_2
 	\begin{pmatrix}
 		e^{-\frac{\iota}{2}(\theta_1^{\alpha_1}+\theta_2^{\alpha_2})}+2\cos(\frac{\theta_1^{\alpha_1}-\theta_2^{\alpha_2}}{2}) &e^{-\frac{\iota}{2}(\theta_1^{\alpha_1}+\theta_2^{\alpha_2})}\\
 		e^{\frac{\iota}{2}(\theta_1^{\alpha_1}+\theta_2^{\alpha_2})}&     e^{\frac{\iota}{2}(\theta_1^{\alpha_1}+\theta_2^{\alpha_2})}+2\cos(\frac{\theta_1^{\alpha_1}-\theta_2^{\alpha_2}}{2})
 	\end{pmatrix}.
 \end{equation*}
 As in the scalar case discussed in \Cref{sec:LFA}, the $10\times 10$ blocks of $\widetilde{R}_h(\bm{\theta^\alpha})$ are assembled into a single block symbol for restriction given by
 \begin{equation*}
 	\widetilde{\boldsymbol{R}}_h(\boldsymbol{\theta})= \begin{pmatrix}  \widetilde{R}_h(\boldsymbol{\theta^{00}}) &  \widetilde{R}_h(\boldsymbol{\theta^{10}}) &  \widetilde{R}_h(\boldsymbol{\theta^{01}}) &  \widetilde{R}_h(\boldsymbol{\theta^{11}}) \end{pmatrix}\in \mathbb{C}^{10\times 40}.
 \end{equation*}
 %
 
 \section{LFA for Vanka Relaxation}\label{secSM:LFA-Vanka}
 As an overlapping additive Schwarz relaxation scheme, the Vanka relaxation considered here takes the current residual, $\bm{r}^{(j)} = \bm{b} - \mathcal{A}^{\text{RQ}}\bm{x}^{(j)}$, and solves the projected system,
 \begin{equation*}
 	\mathcal{A}^{\text{RQ}}_\ell \hat{\bm{x}}_\ell := V_\ell\mathcal{A}^{\text{RQ}}V_\ell^{\top} \hat{\bm{x}}_\ell=V_\ell\vec{r}^{(j)},
 \end{equation*}
 on each block, $\ell$.  This gives a relaxation scheme with error-propagation operator
 \begin{equation*}
 	I - \omega\mathcal{M}^{-1}\mathcal{A}^{\text{RQ}} = I - \omega\left(\sum_{\ell}V_{\ell}^{\top}D_{\ell} (\mathcal{A}^{\text{RQ}}_{\ell})^{-1}V_{\ell}\right)\mathcal{A}^{\text{RQ}}.
 \end{equation*}
 We find the symbol of $\mathcal{M}^{-1}$ by finding the symbols for its components pieces, following the approach presented in \cite{PFarrell_etal_2019a}.  Consider the space of functions represented by a common  Fourier frequency, $\bm{\theta}$.  For the system under consideration here, this is a 10-dimensional space, which is composed of arbitrary linear combinations of the Fourier modes for each DoF type at frequency $\bm{\theta}$.  As such, there is a one-to-one correspondence between functions in this space and vectors in $\mathbb{C}^{10}$.  When Vanka relaxation is applied to a function in the space, the symbol of $\mathcal{M}^{-1}$ acts as a linear map (matrix) from the coefficient vector describing the function before relaxation to that after relaxation.  Each matrix in the definition of $\mathcal{M}^{-1}$ can be understood by its action on that vector of length 10.  An important consequence of this is that, while the definition of $\mathcal{M}^{-1}$ involves a summation over all patches in the mesh, its symbol can be derived by considering the operators only on a single patch.
 
 While the Fourier symbol is, necessarily, a matrix in $\mathbb{C}^{10\times 10}$, the component pieces are larger, given the 20-DoF patch shown at right of \cref{fig:patches}.  Matrix $V_\ell$ maps between vectors on the infinite mesh considered in LFA to those on the patch, simply by selecting the appropriate Fourier coefficients for each DoF type, duplicating the values associated with each edge DoF, and creating 3 copies of each P0 DoF.  The scaling matrix, $D_\ell$, acts directly on these duplicated DoFs, so is its own Fourier representation.  The ``patch matrix'', $V_\ell\mathcal{A}^{\text{RQ}}V_\ell^{\top}$ is easily computed directly as a $20\times 20$ matrix, whose symbol arises by simply accounting for the ``offsets'' between the locations of the DoFs on the mesh, as described in \cite{PFarrell_etal_2019a}.
 %
 \section{LFA for BSR Relaxation}\label{secSM:LFA-BSR}
 The Fourier representation of exact BSR relaxation (see \Cref{subsec:BSR}) comes from that of
 \begin{equation*}
 	\mathcal{M} =
 	\begin{pmatrix}
 		F & B^{\top}\\
 		B & -C
 	\end{pmatrix},
 \end{equation*}
 where $F$ is the approximation of $A$ used in the relaxation scheme, \cref{eqn:BSR_standard}.  Here, we consider
 \begin{equation}\label{eqn:F-Vanka}
 	F = \begin{pmatrix}
 		A_{{\rm{V}},\bm{u}}^{\text{RQ}}  & 0\\
 		0         & \tau D_{\bm w}
 	\end{pmatrix},
 \end{equation}
 where $D_{\bm w}={\rm{diag}}(M_{\bm w})$, the diagonal of the mass matrix on the RT0 space, and $A_{{\rm V},\bm u}^{\text{RQ}}$ is the matrix representation of the additive Vanka relaxation scheme used to approximate the displacement subsystem.  The symbol for $A_{\rm{V},\bm{u}}^{\text{RQ}}$ is found in the same manner as described previously, while that for $D_{\bm w}$ is a $3\times 3$ diagonal symbol whose entries are given by the diagonal entries of the matrix itself.
 
 Since the symbol for $B$ was already derived above, the only remaining calculation is that of the symbol for the approximation to the Schur complement used.  As described in \Cref{subsec:BSR}, a reasonable approximation of the $\bm{u}$ contribution to $BA^{-1}B^\top$ is by a scaled P0 mass matrix, and the $\bm{w}$ contribution can be computed explicitly if we approximation $M_{\bm w}$ by its diagonal.  This leads to the practical approximation of the true Schur complement of $\mathcal{M}$ as
 \[
 S = \left(\frac{1}{M} + \frac{\alpha^2}{\lambda + 2\mu/d}\right)M_p + \tau B_{\bm{w}} D_{\bm w}^{-1}B_{\bm w}^{\top},
 \]
 whose symbol is directly calculated.
 The use of this approximation leads to a slight modification of the matrix representation of the relaxation scheme, writing
 \begin{equation*}
 	\mathcal{M} =
 	\begin{pmatrix}
 		F & B^{\top}\\
 		B & -C_{1}
 	\end{pmatrix},
 \end{equation*}
 where
 \begin{equation*}
 	C_{1} =  \left(\frac{1}{M} + \frac{\alpha^2}{\lambda + 2\mu/d}\right)M_p-\alpha^2 B_{\bm{u}}( A_{{\rm V},\bm{u}}^{\text{RQ}})^{-1}B_{\bm{u}}^{\top},
 \end{equation*}
 arises from subtracting the true contribution to the Schur complement and adding its approximation.  In this form, the symbol of $\mathcal{M}$ is readily computed.
 
 As direct inversion of $S$ is impractical, we consider an inexact variant of Braess-Sarazin relaxation where we use a single sweep of a weighted Jacobi iteration to approximate solution of the linear system with $S$.  The matrix representation of this iteration is given by
 \begin{equation*}
 	\mathcal{M} =
 	\begin{pmatrix}
 		F & B^{\top}\\
 		B & -C_{2}
 	\end{pmatrix},
 \end{equation*}
 where
 \begin{equation*}
 	C_{2} = \frac{1}{\omega_J}{\rm diag}(S) -\alpha^2 B_{\bm{u}}( A_{{\rm V},\bm{u}}^{\text{RQ}})^{-1}B_{\bm{u}}^{\top}-\tau B_{\bm w}(D_{\bm w})^{-1}B_{\bm w}^{\top}.
 \end{equation*}
 The added relaxation parameter, $\omega_J$, can be determined by optimizing the two-grid LFA convergence factor via brute-force or other approaches \cite{LFAoptAlg}.  The symbol of $C_2$ is again easy to derive given symbols for its component parts.

\bibliographystyle{siam}
\bibliography{biot_refs}

\end{document}